\documentclass[12pt]{amsart}
\usepackage[english]{babel}
\usepackage[pdftex]{graphicx}
\usepackage{wrapfig}
\usepackage{pdfsync}
\usepackage{epsfig}
\usepackage{epstopdf}
\usepackage{latexsym,amssymb,amsfonts,amsmath, amscd, euscript, MnSymbol}
\usepackage{enumerate}
\usepackage{verbatim}

\newcommand{\N}{\mathbb{N}}

\newcommand{\Z}{\mathbb{Z}}

\def\Power #1 { \powerset(#1) }
\def\Bidom #1 { {\mathfrak P} (#1) }

\newtheorem{definition}{{\bf Definition}}[section]
\newtheorem{theorem}[definition]{{\bf Theorem}}
\newtheorem{maintheorem}{{\bf Main Theorem}}
\newtheorem{corollary}[definition]{{\bf Corollary}}
\newtheorem{proposition}[definition]{\noindent {\bf Proposition}}
\newtheorem{lemma}[definition]{\noindent {\bf Lemma}}

\newtheorem{claim}[definition]{\noindent {\bf Claim}}
\newtheorem{question}[definition]{\noindent {\bf Question}}
\newtheorem{example}[definition]{\noindent {\bf Example}}

\newtheorem{remark}[definition]{\noindent {\bf Remark}}\newtheorem{problem}[definition]{\noindent {\bf Problem}}
\newtheorem{fact}[definition]{\noindent {\bf Fact}}

\def\proofref #1 {{\noindent  {\bf Proof} (#1).}\ }
\def\downarrownogap{\downarrow \!\!}
\def\uparrownogap{\uparrow \!\!}

\def\endproof{\hfill {\kern 6pt\penalty 500
\raise -0pt\hbox{\vrule \vbox to5pt {\hrule width 5pt
\vfill\hrule}\vrule}}}
\def\centerpicture #1 by #2 (#3){\leavevmode
        \vbox to #2{
        \hrule width #1 height 0pt depth 0pt
        \vfill
        \special{pictfile #3}}}
\title[Convex sublattices and fixed point property]{Convex sublattices of a  lattice \\ and a fixed point property }

\author[D.Duffus]{Dwight Duffus} 
\address{Mathematics \& Computer Science Department, Emory University, Atlanta, Georgia, USA 30322} 
\email{dwight@mathcs.emory.edu}
\author[C.Laflamme]{Claude Laflamme*} 
\address{ Mathematics \& Statistics Department, University of Calgary, Calgary, Alberta, Canada T2N 1N4}
\email{laflamme@ucalgary.ca} 
\thanks{*Supported by NSERC of Canada Grant \# 690404} 
\author[M.Pouzet]{Maurice Pouzet**} \address{ICJ, Math\'ematiques, Universit\'e Claude-Bernard Lyon1, 43 bd. 11 Novembre 1918, 69622 Villeurbanne Cedex, France and Mathematics \& Statistics Department, University of Calgary, Calgary, Alberta, Canada T2N 1N4}
 \email{pouzet@univ-lyon1.fr }
\author[R.Woodrow]{Robert Woodrow} \address{Mathematics \& Statistics Department, University of Calgary, Calgary, Alberta, Canada T2N 1N4}
\email{woodrow@ucalgary.ca }  
\thanks{**Research completed while the author visited the Mathematics and Statistics Department of the University of 
Calgary in July 2008; the support provided is gratefully acknowledged.}

\date{\today }

\begin{document}


\begin{abstract}  
The collection $\mathcal{C}_{L}(T)$ of nonempty convex sublattices of
a lattice $T$ ordered by bi-domination is a lattice. We say that $T$ has
the \emph{fixed point property for convex sublattices} (CLFPP for
short) if every order preserving map $f:T\rightarrow
\mathcal{C}_{L}(T)$ has a fixed point, that is $x\in f(x)$ for some
$x\in T$. We examine which lattices may have CLFPP.  We introduce the
\emph{selection property for convex sublattices} (CLSP); we observe
that a complete lattice with CLSP must have CLFPP, and that this
property implies that $\mathcal{C}_{L}(T)$ is complete.  We show that
for a lattice $T$, the fact that $\mathcal{C}_{ L}(T)$ is complete is
equivalent to the fact that $T$ is complete and the lattice $\powerset
(\omega)$ of all subsets of a countable set, ordered by containment,
is not order embeddable into $T$. We show that for the lattice
$T:=\mathcal {I}(P)$ of initial segments of a poset $P$, the
implications above are equivalences and that these properties are
equivalent to the fact that $P$ has no infinite antichain. A crucial
part of this proof is a straightforward application of a wonderful
Hausdorff type result due to Abraham, Bonnet, Cummings, D\v{z}amondja
and Thompson 2010 \cite{abraham-all}.
\end{abstract} 
\maketitle


\section{Introduction} 

Let $E$ be a set. A \emph{multivalued map} defined on $E$ is a map $f$
from $E$ into $\Power E $, the power set of $E$, and a \emph{fixed
point} of $f$ is an element $x \in E$ such that $x \in f(x)$. The
consideration of the existence of fixed points for various kind of
multivalued maps originates in analysis, with Kakutani's proof of von
Neumann's minmax theorem.  Investigation of fixed points of
multi-valued maps occurs also in the study of partially ordered sets
(\cite{smithson}, see \cite{R}, \cite{walker}), and that is our goal
in this paper.
 
Given a poset $P:= (E, \leq)$, we shall reserve the notation $\Power E
$ for the power set of $E$ equipped with the usual subset
inclusion. We shall use instead $\Bidom E := \Power E \setminus
\{\emptyset\}$ to denote the nonempty subsets  of $E$ endowed with the
following \emph{bi-dominating} preorder: for $A, B \subseteq E$, $A \leq B$ if
every $x \in A$ is below some $y\in B$ and every $y \in B$ is above
some $x \in A$. Note that we removed the empty set for convenience
only, as that subset would be  incomparable to any other subset in the
bi-dominating ordering.  A map $f: E \rightarrow \mathfrak P(E)$ such
that $x\leq y $ implies $f(x)\leq f(y)$ is said to be \emph{order
preserving} (despite the fact that the bi-dominating preorder is not an
order in general).

The poset $P$ has the \emph{relational fixed point} (RFPP for short)
if every order preserving map $f: P\rightarrow \Bidom P$ has a fixed
point. This property and the bi-dominating preorder were introduced by
Walker in \cite{walker} (who used the phrase ``isotone maps" for what we call
``order preserving maps" and gave no name to the bi-dominating preorder), in
connection with the fixed point property for the product of posets.  He
proved that a finite poset $P$ has RFPP if and only if it is
dismantlable. In this paper, we are primarily interested by a special
case of this notion: $P$ is a lattice $T$ and multivalued maps take
values in the set $\mathcal {C}_{L}(T)$ of nonempty convex
sublattices of $T$.  We say that $T$ has the
\emph{fixed point property for convex sublattices} (CLFPP for short)
if every order preserving map from $T$ into $\mathcal {C}_{L}(T)$ has
a fixed point. We concentrate on the following general question:

\begin{question}
Which lattices have  CLFPP?
\end{question}

This property extends the fixed point property for order preserving
maps of a lattice into itself. As is well known, \emph{a lattice $T$
has the fixed property if and only if $T$ is complete \cite {tarski},
\cite{davis}}. Since CLFPP extends FPP, a lattice satisfying  
CLFPP must be complete. More is true: as we will see, $\mathcal{C}_{
L}(T)$ must be complete too (see Proposition
\ref{routineproposition}). It is a tempting conjecture that the
converse holds; we have not yet succeeded in proving or disproving
this.  However, a main result of this paper characterizes lattices $T$
such that $\mathcal{C}_{ L}(T)$ is complete.  This characterization is
based on the nonembeddability of $\powerset (\omega )$, the set of
all subsets of $\omega$, the set of nonnegative integers, ordered by
containment.

\begin{maintheorem} \label{thm:main} 
Let $T$ be a lattice. Then the following properties are equivalent:
\begin{enumerate} [(i)]
\item $\mathcal{C}_{ L}(T)$ is a complete lattice.
\item Every lattice quotient of every retract of $T$ is complete.
\item $T$ is a complete lattice and $\powerset (\omega)$ is not order embeddable in $T$. 
\end{enumerate}
\end{maintheorem}

The implication $(iii)\Rightarrow (ii)$ is, under a seemingly weaker form, in \cite{P-R}.

Complete lattices in which $\powerset (\omega)$ is not order
embeddable arise naturally.  A quite familiar example allows us to
sharpen the preceding characterization for a particular family of
complete lattices. (For more sophisticated examples, see \cite
{L-M-P1, L-M-P2}.)  First, we introduce another property that is
closely related to CLFPP and completeness of $C_L (T)$.

We say that a lattice $T$ has the \emph{selection property for convex
sublattices} (CLSP) if there is an order preserving map $\varphi$ from
$\mathcal{C}_{ L}(T)$ into $T$ such that $\varphi(X)\in X$ for all
$X\in \mathcal{C}_{L}(T)$. For complete lattices, CLSP implies 
CLFPP (Proposition \ref{routineproposition}).

Let us recall that a \emph{closure} on a set $E$ is a function
$\varphi$ defined on the subsets of $E$ such that $X\subseteq
\varphi(X)\subseteq \varphi(Y)= \varphi(\varphi(Y))$ whenever
$X\subseteq Y\subseteq E$. A subset $X$ of $E$ is \emph{closed} if
$\varphi(X)=X$; it is \emph{independent} if $x\not \in \varphi
(X\setminus \{x\})$ for every $x\in X$. The set $\mathcal C_{\varphi}$
of closed sets is a complete lattice.  Here is a well-known fact (see
\cite{chakir-pouzet} Theorem 1.2 or \cite {L-M-P1}).

\begin{fact}\label{fact:closure} 
The lattice $\powerset (\omega)$ is not order embeddable in the
complete lattice $\mathcal C_{\varphi}$ if and only if $E$ contains no
infinite independent set.
\end{fact}
 
If $P$ is an ordered set, the map which associates $\downarrownogap 
X$, the initial segment generated by $X$, to each $X \subseteq P$ is a
closure and the lattice of closed sets is the set $\mathcal {I}(P)$ of
initial segments of $P$, ordered by containment.  A subset $S$ is
independent with respect to this closure if and only if $S$ is an
antichain in $P$.  Hence from Fact \ref{fact:closure}, we have that
\emph{$\powerset (\omega)$ is not order embeddable in $\mathcal {I}(P)$ 
if and only if $P$ contains no infinite antichain}.

\medskip

In general we have the following.

\begin{maintheorem}\label{thm:main2}
Let $P$ be a poset and $T:=\mathcal{I}(P)$ be the lattice of initial
segments of $P$. Then the following properties are equivalent:
\begin{enumerate}[{(i)}]
 \item $T$ has  CLSP;
 \item $T$ has  CLFPP;
 \item $\mathcal{C}_{ L}(T)$ is a complete lattice;
\item $P$ has no infinite antichain.
\end{enumerate}
\end{maintheorem}

The implications $(i)\Rightarrow (ii)$ and $(ii)\Rightarrow (iii)$,
included in Proposition \ref{routineproposition}, are
easy. The implication $(iii)\Rightarrow (iv)$ mixes the easy part of
Theorem \ref{thm:main}, namely implication $(i)\Rightarrow (iii)$,
with Fact \ref{fact:closure}. The implication $(iv)\Rightarrow (i)$, the
content of Lemma \ref {lem:key}, is the cornerstone of the theorem.
It follows from a Hausdorff type result due to Abraham, Bonnet,
Cummings, D\v{z}amondja and Thompson 2010 \cite{abraham-all}.

From previous results in \cite{D-R-S} and \cite{pou-riv84}, we know
that CLFPP holds for countable complete lattices and finite
dimensional complete lattices.  In Section \ref{sublattices} we look
at the relationship between various properties which may hold for the
class of lattices with  CLSP, the class of lattices with CLFPP,
as well as the class of lattices such that $\mathcal{C}_{L}(T)$ is
complete. Results and problems are reviewed at the end of Section
\ref{conclusion}.

In Section \ref{subsets}, we consider also maps defined on a poset $P$
and whose values are convex subsets of $P$. We establish the exact
relationship between FPP for those maps (CFPP) and RFPP, which yields the following:

\begin{maintheorem}\label{RFPP}
A poset has RFPP if and only if it has CFPP and contains no chain of
the same type as the integers. In particular RFPP and CFPP coincide for finite posets.
\end{maintheorem}

We illustrate the differences with CLFPP, the main illustration
being Theorem \ref{thm:illus}.

In the following section, we present the required notation,
terminology and preliminary results.

\section{Preliminaries}

Our notation is standard apart from one instance.  Let $f$ be a map with domain a set $E$.
For $X \subseteq E$, let $f[X] = \{ f(x) : x \in X \}$, rather than the more common $f(X)$.
We denote by $\omega$
the order type of the set $\N$ of non negative integers and use $n\in
\N$ as well as $n \in \omega$ and $n<\omega$.  Similarly  we denote
by $\theta$ the order type of the chain $\Z$ of integers. We recall
below some basic notions of the theory of ordered sets, but we refer to
\cite {jech} for other undefined set theoretical notation.  

Let $P: =(E, \leq)$ be a poset. In the sequel we will denote by $P$
the set $E$; that is we write $A \subseteq P$ to mean $A \subseteq
E$. Let $A$ be a subset of $P$.  We say that $A$ is an {\it initial
segment} of $P$ if $x\in P$, $y\in A$ and $x\leq y$ imply $x\in A$. We
say that $A$ is {\it up-directed} if every pair of elements of $A$ has
a common upper bound in $A$. An {\it ideal} of $P$ is a nonempty
up-directed initial segment of $P$.  A {\it final segment}, resp. a
down directed subset, resp.  \emph{filter}, of $P$ is any initial
segment, resp. up-directed subset, resp. ideal, of $P^*$, the dual of
$P$.

We use this notation and terminology:

\begin{itemize}
\item $U(A) :=\{x\in P: y\leq x \; \text{ for all } \; y\in A\}$ is the
set of upper bounds of $A$ and $L(A)$, the set of lower bounds, is
defined dually; 
\item $\downarrownogap  A := \{x\in P: x\leq y\; \text{for some} \; y\in A\}$ 
is the \emph{initial segment generated by} $A$ and $\uparrownogap A$
is defined dually and called the \emph{final segment generated by} $A$.
\end{itemize}

For a singleton $x\in P$, we use $\downarrownogap  x$ instead of $\downarrownogap  \{x\}$. 
If reference to $P$ is needed, particularly in case of several orders on
the same ground set, we may use the notation $\downarrow_{P} \!\! A$.

A subset $A$ of $P$ is \emph{cofinal}, resp. \emph{coinitial}, in $P$
if $\downarrownogap A=P$, resp. $\uparrownogap A=P$. The \emph{cofinality} of
$P$, denoted by $cf(P)$ is the least cardinal $\kappa$ such that $P$
contains a cofinal subset of size $\kappa$.

We now review the basic notions associated with gaps. Let $(A,B)$ be a
pair of subsets of $P$.

\begin{itemize} 
   \item $(A,B)$ is a  \emph{pregap} of $P$ if $A\subseteq L(B)$
  or, equivalently, if $B\subseteq U(A)$;
  \item $(A, B)$ is \emph{totally ordered} if $A\cup B$ is totally ordered via the order on $P$;  
  \item If $(A, B)$ is a pregap, set $S(A,B):=U(A)\cap L(B)$;
  \item A pregap $(A,B)$ is called \emph{separable} if $S(A,B)\ne \emptyset$;
  \item A pregap $(A,B)$ is called a \emph{gap} if $S(A, B) = \emptyset$.
  \end{itemize}
  
Concerning the separation of pregaps, we recall that a poset is a
complete lattice iff every pregap is separated.  As it is easy to see
a lattice is complete iff every pregap where $A$ is up-directed and
$B$ down-directed is separated.  It is known that this later condition
is equivalent to the fact that every totally ordered pregap is
separated, see \cite{D-R}.
  
We denote by $ \mathcal{I}(P)$, resp. $\mathcal {I}d (P)$, the set of
initial segments, resp. ideals of $P$ ordered by inclusion. We denote
by $\mathcal {F}(P)$, resp. $\mathcal {F}i(P)$, the set of final
segments, resp. filters, of $P$ ordered by inclusion.

Let $A$ be subset of $P$.  We say that $A$ is \emph{convex} if $x,
y\in A$, $z\in P$ and $x\leq z\leq y$ imply $z\in A$. Such sets are
also called \emph{intervals}. If $a, b\in P$ we set $[a, b]:= \{x\in
P: a\leq x\leq b\}$. This set is a \emph{closed interval}. A nonempty
set is of this form if and only if it is convex with a least and
largest element. A subset $A$ of $P$ is convex if and only if $A:=
I\cap F$ for some pair $(I,F)\in \mathcal{I}(P)\times \mathcal
{F}(P)$.  In particular, the set $Conv_P(A):=\downarrownogap  A\cap 
\uparrownogap A$ (=$\{z\in P: x\leq z\leq y\; \text{for some}\; (x,y)\in A^2\}$) is
convex and called the \emph{convex envelope} of $A$.  The set
 of convex subsets of $P$ is a closure system, and
the convex envelope of $A$ is nothing else than the closure of
$A$. Ordered by inclusion, this set  is a complete lattice (in
fact an algebraic lattice) which is the object of several studies eg
\cite{birkhoff}, \cite{semenova}. 
We denote the set on nonempty convex subsets of $P$  by $\mathcal{C}(P)$.
In this paper we are concerned with $\mathcal{C}(P)$ under a different ordering, which is the subject
of the next section.

\section{The poset of convex subsets of a poset}\label{subsets}


Let $P$ be a poset. As noted in the introduction, we use $\Bidom P :=
\Power P \setminus \{\emptyset\}$ to denote the nonempty subsets  of (the domain of) $P$ endowed with the
following \emph{bi-dominating} preorder.   For  $ X, Y \in \Bidom  P$ define:
\begin{equation}\label{eq:main1}
X\leq Y \; \text{if}\; X\subseteq \downarrownogap  Y  \:\text{and} \:Y \subseteq \uparrownogap X. 
\end{equation}
We also set:
\begin{equation}\label{eq:main2}
X\equiv Y \; \text{if}\; X \leq Y \:\text{and} \:Y \leq X. 
\end{equation}

\begin{lemma} 
The relation $\leq$ on $\Bidom P$ is a preorder. It induces
an order on $\mathcal C(P)$ which is isomorphic to the quotient of
this preorder by the equivalence associated with the preorder.
\end{lemma}

\begin{proof}  
The first part of this lemma is obvious. The second part relies on the
following facts both of which are straightforward to prove:\\[-.2cm]

\qquad $Conv(X) \equiv X$, and\\[-.2cm]

\qquad $X\equiv Y \Leftrightarrow Conv (X)=Conv (Y)$\\[-.2cm]
 
for all $X, Y \in \Bidom P$.
\end{proof}

\vspace{.3cm}

Note that we could have defined this preorder on the collection of all
subsets, but then the empty set would have been incomparable to every
nonempty subset. Therefore it is avoided.

The poset $\mathcal C(P)$ has a simple representation:

\begin{fact} \label{fact3.1} 
The map $\phi$ from $\mathcal {C}(P)$ to the product $\mathcal
{I}(P)\times \mathcal {F}(P)^{*}$, defined by
$\phi(A):=(\downarrownogap A, \uparrownogap A)$ is an embedding.
\end{fact} 

\begin{remark} 
Let $X,Y\in \mathcal C(P)$. If $Y\subseteq X$ then $X\leq Y$ if and
only if $X\subseteq \downarrownogap Y$, that is $Y$ cofinal in
$X$. Consequently, \emph{a convex subset of $P$ is above $P$
w.r.t. the bi-domination order if and only if this is a final segment
of $P$ which is cofinal in $P$}. Subsets of a poset with these
properties are called \emph{open dense sets} and are well known in the
literature in particular regarding Baire spaces. 
\end{remark}



\begin{definition}
Let $P$ be a poset. We say that $P$ has the \emph{convex fixed point
property} (CFPP for short) if every order preserving map
$f:P\rightarrow \mathcal C(P)$ has a fixed point.
\end{definition}


The relationship between RFPP and CFPP is quite simple and was given
in Main Theorem \ref{RFPP}.

type $\theta$. In particular RFPP and CFPP coincide for finite posets.

The proof is related to ideas due to Walker. The
first involves the notion of retraction, where a poset $Q$ is called a
\emph{retract} of another poset $P$ if there are order preserving maps
$s:Q\rightarrow P$ and $r: P\rightarrow Q$ such that $r\circ s= 1_Q$.
The maps $r$ and $s$ are called a \emph{retraction} and a
\emph{coretraction}.

\begin{lemma} \label{lem:retractRFPP}
RFPP is preserved under retracts. 
\end{lemma}

 
\begin{proof} 
Let $P$ be a poset satisfying RFPP, $Q$ an order retract of $P$, and
$g: Q \rightarrow \Bidom Q$ an order preserving map. Let $s: Q
\rightarrow P$ and $r:P\rightarrow Q$ be two order preserving maps
such that $r\circ s =1_{Q}$.  The map $h: P\rightarrow \Bidom P$
defined by $h(x):= s[g(r(x)]$ is order preserving, hence it has a
fixed point, say $u$.  Since $u\in h(u)$ there is some $y\in g(r(u))$
such that $u=s(y)$. We have $r(u)=r(s(y))=y$, thus $y\in g(y)$ and
hence $y$ is a fixed point of $g$.
\end{proof}
 
\medskip

The second idea involves the notion of well foundedness.  

\begin{definition}
We recall that a poset $P$ is \emph{well founded} if every nonempty
subset contains some minimal element; equivalently, $P$ contains no
infinite chain of type $\omega^*$. 
\end{definition}

The next lemma is essentially Proposition 5.2 of \cite {walker}.

\begin{lemma}\label{lem:wellfounded} 
Le $P$ be a poset and $f:P\rightarrow \Bidom P$  be a
preorder preserving map. If there exists some $x\in P$ such that
$\downarrownogap x$ is well founded and meets $f(x)$ then $f$ has a
fixed point. 
\end{lemma}

\begin{proof}
Define by induction a sequence $(x_n)_{n<\omega}$ of elements of
$P$. Let $m<\omega$ and suppose $(x_n)_{n<m}$ be defined. If $m=0$, set
$x_m:= x$. Otherwise, choose $x_m\in f(x_{m-1})\cap \downarrownogap
x_{m-1}$. This sequence is well defined. It is descending, thus
stationary, hence it yields a fixed point.

\end{proof}

\medskip

\noindent
{\bf Proof of Theorem \ref{RFPP}.}  Let $P$ be a poset. Suppose that
$P$ has RFPP.  Trivially, it has CFPP. Suppose that it contains a
chain $Z$ of type $\theta$. Extend this chain to a maximal chain
$C$. As a maximal chain of $P$, this is a retract of $P$ by \cite
{D-R-S}. Since RFPP is preserved under retraction (by Lemma
\ref{lem:retractRFPP}), $C$ has RFPP. In particular, $C$ has FPP,
hence $C$ is complete. Hence $Z$ has an infimum $a$ and a supremum
$b$. As it is easy to see, the chain $D:=\{a\}\cup Z\cup \{b\}$ is a
retract of $C$, thus it has RFPP. But this is trivially false. Indeed,
let $f: D\rightarrow
\Bidom D$ defined by $f(x):=Z\setminus\{x\}$ if $x\in Z$ and
$f(x):= Z$ if $x\in \{a,b\}$. This map is order preserving but has no
fixed point. Thus $P$ cannot contain a chain of type
$\theta$. 

\medskip
Conversely, suppose that $P$ has CFPP and contains no chain of type
$\theta$. Let $f: P\rightarrow \Bidom P$ be an order preserving
map. Let $\overline f: P\rightarrow \mathcal C (P)$ defined by setting
$\overline f(x):= Conv_P(f(x))$. This map is order preserving
too. Thus it has a fixed point $x$. Since $x\in \overline
f(x):=Conv_P(f(x))$ there are $u, v\in f(x)$ such that $u\leq x \leq
v$. Since $P$ contains no chain of type $\theta$, either
$\downarrownogap x$ is well founded or $\uparrownogap x$ is dually
well founded. W.l.o.g we may suppose that $\downarrownogap x$ is well
founded (otherwise consider $P^*$). Apply Lemma
\ref{lem:wellfounded} and we obtain that $f$ has a fixed point.
\endproof

\subsection{FPP for the poset of convex subsets of a poset}

Since the set $\mathcal C (P)$ of nonempty convex subsets of a
poset $P$ is a poset, a straightforward question emerges:

\begin{question}\label{question1b}
How can we relate  FPP for   $\mathcal C (P)$ and  CFPP for $P$?

\end{question}

This simple minded question is at the root of this paper. As we will
see, there is no relation in general. There are posets $P$ without
CFPP for which $\mathcal C (P)$ has FPP; a straightforward example is
the ordinal sum of two $2$-element chains (see Example
\ref{examplenondismantlable}). And, on an other hand, there are posets
$P$ with CFPP but such that $\mathcal C (P)$ does not have FPP (see
Lemma \ref{lem:nofixedpoint}). According to Theorem \ref{thm:not FPP}
below these posets are infinite.

The example we give in Lemma \ref {lem:nofixedpoint} relies on the
well-known fact that a poset containing a totally ordered gap does not
have FPP. With that in hand, we construct a complete lattice
$\overline Q$ with CFPP such that $\mathcal C (\overline Q)$
contains a totally ordered gap.

We start our discussion with some simple facts about pregaps, namely a
necessary condition for separation and the fact that each gap yields a
gap having a special form.

Let $P$ be a poset.  If $\mathcal A \subseteq \mathcal C (P)$, we
set $I_{\mathcal A}:= \bigcap\{\downarrownogap A: A\in \mathcal A\}$
and $F_{\mathcal A}:= \bigcap\{\uparrownogap A: A\in \mathcal A\}$. If
$(\mathcal A, \mathcal B)$ is a pregap of $\mathcal C (P)$ we set
$\mathcal A_{\mathcal B}:= \{ I_{\mathcal B}\cap \uparrownogap A: A\in
\mathcal A\}$ and $\mathcal B_{\mathcal A}:= \{F_{\mathcal A} \cap
\downarrownogap B : B\in \mathcal B\}$.

\begin{lemma}\label{lem:pregaps}
Let $(\mathcal A,  \mathcal B)$ be a pregap of $\mathcal C (P)$. 
\begin{enumerate} [{(i)}] 

\item \label{item i} 
If $(\mathcal A, \mathcal B)$ is separable then $F_{\mathcal A}\cap
I_{\mathcal B}$ separates it; furthermore it contains every separator.
In particular $F_{\mathcal A}\cap I_{\mathcal B}$ is nonempty.
 
\item \label{item ii}
$(\mathcal A_{\mathcal B}, \mathcal B_{\mathcal A})$ is a pregap of
$\mathcal C (P)$ such that $F_{\mathcal A_{\mathcal
B}}=F_{\mathcal A}$, $I_{\mathcal B_{\mathcal A}}= I_{\mathcal B}$ and
$S(\mathcal A_{\mathcal B}, \mathcal B_{\mathcal A})\subseteq
S(\mathcal A,\mathcal B)$. In particular $(\mathcal A_{\mathcal B},
\mathcal B_{\mathcal A})$ is a gap whenever $(\mathcal A, \mathcal B)$
is a gap.
 
\end{enumerate}
\end{lemma}

\begin{proof} 
Item (i). Set $Z:=F_{\mathcal A}\cap I_{\mathcal B}$. Let $C\in
S(\mathcal A, \mathcal B)$. For every $A\in \mathcal A$, $B\in
\mathcal B$ we have $C\subseteq \uparrownogap A$ and $C\subseteq
\downarrownogap B$ hence $C\subseteq Z$. Let $A\in \mathcal A$. We
have trivially $Z\subseteq \uparrownogap A$. We have $A\subseteq
\downarrownogap C$ and $C\subseteq Z$ hence $A\subseteq
\downarrownogap Z$ thus $A\leq Z$. By the same token, we have $Z\leq
B$ if $B\in \mathcal B$ hence $Z \in S(\mathcal A, \mathcal B)$.

Item (ii). We first prove two basic claims. Let $A\in \mathcal A$ and $B\in \mathcal B$. 

\begin{claim} \label{claim:gap1} 
$\uparrownogap (I_{\mathcal B}\cap \uparrownogap A)=\uparrownogap A$
and $\downarrownogap (F_{\mathcal A}\cap \downarrownogap
B)=\downarrownogap B$.  
\end{claim}

\begin{proof}
We prove only the first equality. We have $A \subseteq I_{\mathcal B}$
since $(\mathcal A, \mathcal B)$ is a pregap.  This yields $
A\subseteq \uparrownogap (I_{\mathcal B}\cap \uparrownogap A)$. Thus
$\uparrownogap A\subseteq \uparrownogap (I_{\mathcal B}\cap
\uparrownogap A)$. The reverse inclusion holds since $I_{\mathcal
B}\cap \uparrownogap A$ is a subset of $\uparrownogap A$.
\end{proof}

\begin{claim} \label{claim:gap2} 
$A\leq I_{\mathcal B}\cap \uparrownogap A\leq F_{\mathcal A}\cap
\downarrownogap B \leq B$.
\end{claim}

\begin{proof}
We have trivially $ I_{\mathcal B}\cap \uparrownogap A\subseteq
\uparrownogap A$. Since, as seen in Claim \ref{claim:gap1} above,
$A\subseteq I_{\mathcal B}$ we have $A\subseteq \downarrownogap (
I_{\mathcal B}\cap \uparrownogap A)$. The inequality $A\leq
I_{\mathcal B}\cap \uparrownogap A$ follows. Similarly, we have
$F_{\mathcal A}\cap \downarrownogap B \leq B$. Finally, from $\mathcal
I_{\mathcal B}\subseteq \downarrownogap B$ and Claim \ref{claim:gap2}
above, we have $I_{\mathcal B}\cap \uparrownogap A\subseteq
\downarrownogap B=\downarrownogap ( F_{\mathcal A}\cap \downarrownogap
B)$. Similarly, we have $(F_{\mathcal A}\cap \downarrownogap
B)\subseteq \uparrownogap (I_{\mathcal B}\cap \uparrownogap A)$,
proving $I_{\mathcal B}\cap \uparrownogap A\leq F_{\mathcal A}\cap
\downarrownogap B$.
\end{proof}

Now back to the proof of (ii), the equality $F_{\mathcal A_{\mathcal
B}}=F_{\mathcal A}$ follows from the first part of Claim \ref
{claim:gap1}; the equality $I_{\mathcal B_{\mathcal A}}= I_{\mathcal
B}$ follows from the second part. From Claim \ref{claim:gap2}, we have
$S(\mathcal A_{\mathcal B}, \mathcal B_{\mathcal A})\subseteq
S(\mathcal A,\mathcal B)$.
\end{proof}

\medskip

These particular gaps will play an important role and thus deserve a
dedicated nomenclature. 

\begin{definition} 
A pregap $(\mathcal A', \mathcal B' )$ of the form $(\mathcal
A_{\mathcal B}, \mathcal B_{\mathcal A})$ will be called
\emph{special}.
\end{definition} 

The following result on special pregaps relies on the simple fact that
a nonempty final segment of an up-directed poset is cofinal in that
poset.

\begin{lemma}\label{lem:pregapsbis}
Let $(\mathcal A', \mathcal B')$ be a special pregap of
$\mathcal C (P)$ and let $(\mathcal A, \mathcal B)$ such that
$(\mathcal A', \mathcal B')= (\mathcal A_{\mathcal B}, \mathcal
B_{\mathcal A})$. If $I_{\mathcal B}$ and $F_{\mathcal A}$ are
respectively up and down directed, then

\begin{enumerate} [{(i)}]
\item The order on $\mathcal A'$ coincides with the reverse of the inclusion and the order on $\mathcal B'$ coincides with the inclusion. 
\item   The map $A\hookrightarrow I_{\mathcal B} \cap \uparrownogap A$ from $\mathcal A$ onto $\mathcal A'$ and the map  $B\hookrightarrow F_{\mathcal A}\cap \downarrownogap B$ from $\mathcal B$ onto $\mathcal B'$ are two order preserving maps. 
\item $\downarrownogap A'=I_{\mathcal B'}$ and $\uparrownogap  B'=F_{\mathcal A'}$ for all $A'\in \mathcal A'$, $B'\in \mathcal B'$;
\item The following properties are equivalent:

\begin{enumerate} 
\item $(\mathcal A', \mathcal B')$ is a gap; 
\item $F_{\mathcal A'}\cap I_{\mathcal B'}=F_{\mathcal A}\cap I_{\mathcal
 B}=\emptyset$; \item $(\mathcal A, \mathcal B)$ is a gap.
\end{enumerate}

\end{enumerate}
\end{lemma}

\begin{proof} We begin with a claim. 

\begin{claim} \label{claim:gap3} 
If $I_{\mathcal B}$ is up-directed then $\downarrownogap (I_{\mathcal
B}\cap \uparrownogap A)= I_{\mathcal B}$ for every $A \in \mathcal A$.
\end{claim}

\begin{proof}
Indeed, let $A\in \mathcal A$. Since $(\mathcal A, \mathcal B)$ is a
pregap we have $A \subseteq I_{\mathcal B}$. Hence, the final segment
$ I_{\mathcal B} \cap \uparrownogap A$ of $I_{\mathcal B}$ is
nonempty. Since $I_{\mathcal B}$ is up directed, $ I_{\mathcal B}
\cap \uparrownogap A$ is cofinal in $I_{\mathcal B}$ that is
$\downarrownogap (I_{\mathcal B} \cap \uparrownogap A)= I_{\mathcal
B}$, proving our claim.
\end{proof}

Item (i). Let $A, A' \in \mathcal A$. Suppose $I_{\mathcal B} \cap
\uparrownogap A\leq I_{\mathcal B} \cap \uparrownogap
A'$. Necessarily, we have $I_{\mathcal B} \cap \uparrownogap A'
\subseteq \uparrownogap (I_{\mathcal B} \cap \uparrownogap A)$. It
follows that $I_{\mathcal B} \cap \uparrownogap A' \subseteq
I_{\mathcal B} \cap \uparrownogap A$. Conversely, suppose $I_{\mathcal
B} \cap \uparrownogap A' \subseteq I_{\mathcal B} \cap \uparrownogap
A$. Then trivially, $I_{\mathcal B} \cap \uparrownogap A' \subseteq
\uparrownogap (I_{\mathcal B} \cap \uparrownogap A)$. With Claim
\ref{claim:gap3} we have $(I_{\mathcal B} \cap \uparrownogap
A)\subseteq I_{\mathcal B}=I_{\mathcal B} \cap \uparrownogap A'$,
hence $I_{\mathcal B} \cap \uparrownogap A\leq I_{\mathcal B} \cap
\uparrownogap A'$ as required.

Item (ii).  Now, let $A, A' \in \mathcal A$. According to Claim
\ref{claim:gap3}, $\downarrownogap (I_{\mathcal B}\cap \uparrownogap
A)=I_{\mathcal B}=\downarrow(I_{\mathcal B}\cap \uparrownogap A')$.
Suppose $A\leq A'$.  Then, in particular $\uparrownogap A'\subseteq
\uparrownogap A$ hence $I_{\mathcal B}\cap \uparrownogap A'\subseteq
\uparrownogap (I_{\mathcal B} \cap \uparrownogap A)$. The inequality
$I_{\mathcal B} \cap \uparrownogap A\leq I_{\mathcal B} \cap
\uparrownogap A'$ follows.  Hence, the map $A\hookrightarrow
I_{\mathcal B} \cap A$ is order preserving as claimed. Since
$F_{\mathcal A}$ is down directed, the same property holds for the map
$B\hookrightarrow F_{\mathcal A}\cap \downarrownogap B$.

Item (iii) According to Claim \ref{claim:gap3}, $\downarrownogap
A'=I_{\mathcal B}$; since $I_{\mathcal B}=I_{\mathcal B'}$ we have
$\downarrownogap A'=I_{\mathcal B'}$. Similarly, $\uparrownogap
B'=F_{\mathcal A'}$ for every $B'\in \mathcal B$ and thus (iii) holds.

%

Item (iv). We prove the implications $(a) \Rightarrow (b)\Rightarrow
(c) \Rightarrow (a)$. Suppose that $(b)$ does not hold. Since
$I_{\mathcal B'}=I_{\mathcal B}$ and $F_{\mathcal A'}=F_{\mathcal A}$,
this amounts to $F_{\mathcal A}\cap I_{\mathcal B}\not =\emptyset$. We
claim that $F_{\mathcal A}\cap I_{\mathcal B}$ separates $(\mathcal
A', \mathcal B')$, that is $I_{\mathcal B}\cap \uparrownogap A\leq
F_{\mathcal A}\cap I_{\mathcal B}\leq F_{\mathcal A}\cap
\downarrownogap B$ for all $A\in \mathcal A$ and $B\in\mathcal B$,
hence $(a)$ does not hold. Indeed, let $A\in \mathcal A$.  With Claim
\ref {claim:gap1}, we have $F_{\mathcal A}\cap I_{\mathcal B}\subseteq
\uparrownogap A=\uparrownogap (I_{\mathcal B}\cap \uparrownogap
A)$. Also, $F_{\mathcal A}\cap I_{\mathcal B}$ is a nonempty final
segment of $I_{\mathcal B}$. This later set being up-directed,
$\downarrownogap (F_{\mathcal A}\cap I_{\mathcal B})=I_{\mathcal
B}$. Since $\downarrownogap (I_{\mathcal B}\cap \uparrownogap A)=
I_{\mathcal B}$ from Claim \ref {claim:gap3} it follows that
$I_{\mathcal B}\cap \uparrownogap A\subseteq \downarrownogap
(F_{\mathcal A}\cap I_{\mathcal B})$. Thus $I_{\mathcal B}\cap
\uparrownogap A\leq F_{\mathcal A}\cap I_{\mathcal B}$. The proof that
$F_{\mathcal A}\cap I_{\mathcal B}\leq F_{\mathcal A}\cap
\downarrownogap B$ for $B\in\mathcal B$ is similar.  Implication
$(b)\Rightarrow (c)$ is the contraposition of Item (\ref {item i}) of
Lemma \ref{lem:pregaps}. Implication $(c)\Rightarrow (a)$ is contained
in Item (\ref {item ii}) of Lemma \ref{lem:pregaps}. \end{proof}

\begin{lemma} \label {lem:fix-pointfree} 
If $\mathcal C (P)$ contains a totally ordered pregap $(\mathcal
A,\mathcal B)$ such that $F_{\mathcal A}\cap I_{\mathcal B}=\emptyset$, 
then $P$ does not have CFPP.
\end{lemma}

\begin{proof}
W.l.o.g we may suppose that $\mathcal A:=\{A_{\alpha}:\alpha<\mu\}$
and $\mathcal B: =\{B_{\beta}:\beta<\lambda\}$ are a well ordered
chain and a dually well ordered chain in $\mathcal C (P)$
satisfying: $A_{\alpha}<A_{\gamma}$ if and only if $\alpha<\gamma<\mu$
and $B_{\delta}<B_{\beta}$ if and only if
$\beta<\delta<\lambda$. Define $f: P\rightarrow \mathcal C (P)$ as
follows. Let $x\in P$. If $x\not \in I_{\mathcal B}$, set $f(x):=
B_{\beta}$ where $\beta$ is minimum such that $x\not \in
\downarrownogap B_{\delta}$. If $x\in I_{\mathcal B}$ then $x \not \in
F_{\mathcal A}$ and we set $f(x):= A_{\alpha}$ where $\alpha$ is
minimum such that $x\not \in \uparrownogap A_{\alpha}$. It is easy to
see that this map is order preserving. By construction it has no fixed
point, hence $P$ does not have CFPP.  
\end{proof}

\begin{remark} 
The set $\mathcal C (P)$ of (nonempty) convex subsets of a poset $P$ may
contain a gap $(\mathcal A,\mathcal B)$ which is totally ordered and
such that $F_{\mathcal A}\cap I_{\mathcal B}$ is nonempty. In this
case, $\mathcal C (P)$ does not have FPP.  But, one cannot use this gap to
construct a fixed point free map from $P$ to $\mathcal C (P)$ as in
Lemma \ref{lem:fix-pointfree} above (See Example
\ref{examplenotFPP}). 
\end{remark}

\medskip

The following fact follows from Lemma \ref{lem:wellfounded}:

\begin{fact} \label{wellfounded}
A well founded poset with a largest element has  CFPP.
\end{fact}

\medskip

\begin{definition}
We recall that a poset $P$ is a \emph{well-quasi-order} (wqo for
short) if it contains no infinite antichain and no infinite descending
chain. 
\end{definition}

A well-known result of Higman \cite{higman} shows that if $P$ is
well-quasi-ordered, then $\mathcal {I}(P)$ is well founded. Here we
have the following.

\begin{fact}\label{fact:wqo} 
If $P$ is wqo then $\mathcal C (P)$ is well founded.
\end{fact}

\begin{proof} 
According to Fact \ref{fact3.1}, $\mathcal C (P)$ is embeddable
into the direct product $\mathcal{I}(P)\times \mathcal
F(P)^{*}$. Since $\mathcal F(P)^{*}$ is isomorphic to
$\mathcal{I}(P)$, each factor of this product is well founded. It
turns out that the product is well founded, and that its subsets are
well founded too.
\end{proof}

\begin{lemma}\label{lem: wqo+C(P)} 
If $P$ is wqo with a  largest element, then $\mathcal C (P)$ has FPP. 
\end{lemma}

\begin{proof} 
Let $Q:= \mathcal C (P)$. Then $Q$ has a largest element (namely
$\{a\}$ where $a$ is the largest element of $P$) and is well founded
(Fact \ref{fact:wqo}). Thus it has FPP.
\end{proof}

\vspace{.3cm}

As the following example shows, the well foundedness of $P$ is not enough in Lemma \ref{lem: wqo+C(P)}. 

\begin{example}\label{examplenotFPP} 
Let $F:= \{a,b,c\}$ be a three element set and $Q:= F\times \N$.  For
$x\in F$ set $x_n:= (x, n)$. Order $Q$ in such a way that $a_n<b_n$,
$c_m<b_n$ and $c_m\leq c_n$ for all $m\leq n$. With a top $0$ and
bottom $1$ added to $Q$, the resulting poset $\overline Q$ is a
lattice, in fact a complete lattice.  
\end{example}
 
\begin{claim}\label{lem:nofixedpoint}
The complete lattice $\overline Q$ has CFPP but $\mathcal
C (\overline Q)$ does not have FPP.
\end{claim}

\begin{proof}
$\overline Q$ has a largest element and is well founded, thus it has
CFPP (Fact \ref{wellfounded}).  $\mathcal C (\overline Q)$ does not
have FPP because it contains an $(\omega, \omega^*)$-gap. 

Indeed let:
\[\begin{array}{ll}
A=\{a_n : n \in \N\}, & A_{\geq n} = \{ a_m: m \geq n\} \\
B=\{b_n : n \in \N\}, & B_{\geq n} = \{ b_m: m \geq n\} \\
C=\{c_n : n \in \N\}, & C_{\geq n} = \{ c_m: m \geq n\} \\
F_n = \uparrownogap (A \cup \{c_n\}), & F= \bigcap F_n \\
I_n = \downarrownogap (A \cup B_{\geq n}), & I= \bigcap I_n \\
\end{array}\]

Then a straightforward calculation yields: 
\[\begin{array}{ll}
F_n = A \cup B \cup C_{\geq n} \cup \{1\}  & F = A \cup B \cup \{1\} \\
I_n = A \cup B_{\geq n} \cup C  \cup \{0\}  & I = A \cup C \cup \{0\} \\
\end{array}\]

Now let $A_n=F_n\cap I = A \cup C_{\geq n}$ and $B_n=I_n\cap F = A
\cup B_{\geq n}$, and define $\mathcal G:=(\mathcal A, \mathcal B)$
where $\mathcal A:= \{A_n : n\in\N\}$, $\mathcal B:= \{B_n: n\in\N\}$.
We have $A_n<A_{n+1}\leq B_{m+1}<B_m$ for all $n, m\in N$, hence
$\mathcal G$ is a totally ordered pregap; we also have
$\downarrownogap B_n=I_n$, $\uparrownogap A_n=F_n$, hence $F_{\mathcal
A}=F$, $I_{\mathcal B}= I$ and $F_{\mathcal A}\cap I_{\mathcal B}=A$.
If $\mathcal G$ was a separable pregap then $F_{\mathcal A}\cap
I_{\mathcal B}$ would separate it (Lemma \ref{lem:pregaps}); since
$A_n \not\leq A$ this is not the case. Thus $\mathcal G$ is a gap.

\end{proof}
 
Fact \ref {wellfounded} and Lemma \ref {lem: wqo+C(P)} immediately
yield that \emph{a finite poset $P$ with a largest element has CFPP and
$\mathcal C (P)$ has FPP}. As we will see in Theorem \ref{thm:not
FPP} this property extends to dismantlable posets.  For that, we need
some properties of retracts.

\begin{lemma}\label{lem:C(P)retraction} 
If $Q$ is an order retract of $P$ via the maps $s: Q \rightarrow P$
and $r:P\rightarrow Q$ then $\mathcal C (Q)$ is an order retract of
$\mathcal C (P)$ via the maps $\overline s: \mathcal C (Q) \rightarrow
\mathcal C (P)$ and $\overline r:\mathcal C (P)\rightarrow \mathcal C
(Q)$ defined by $\overline s(Y):= Conv_P(s[Y])$ and $\overline r(X):=
Conv_Q(r[X])$ for all $Y \in \mathcal C (Q)$ and $X\in \mathcal C (P)$.
\end{lemma}

\begin{proof} 
As it is easy to check, the maps $\overline s$ and $\overline r$ are
order preserving. To conclude it suffices to prove that $\overline
r\circ \overline s$ is the identity on $\mathcal C (Q)$. Let $Y\in
\mathcal C (Q)$. Since $s[Y]\subseteq Conv_P(s[Y])$ we have
$Y=r[s[Y]] \subseteq r[Conv_P(s[Y])] \subseteq Conv_Q(r[Conv_P(s[Y])])
= \overline r(\overline s(Y))$. Observing that $r[Conv_P(Z)] \subseteq
conv_Q(r[Z])$ for every subset $Z$ of $P$, we have $r[Conv_P(s[Y])]
\subseteq Conv_Q(r[s[Y])]) = Conv_Q[Y] = Y$ and since $Y$ is convex
$\overline r(\overline s(Y)) = Conv_Q(r[Conv_P(s[Y])]) \subseteq
Y$. Thus $\overline r(\overline s(Y)) = Y$.
\end{proof}

\vspace{.3cm}

Since FPP is preserved under retraction, we have immediately the following: 

\begin{corollary}\label{lem:FPPretraction} 
If $Q$ is an order retract of $P$ and $\mathcal C (P)$ has FPP then
$\mathcal C (Q)$ has FPP.
\end{corollary}

\begin{corollary}\label{lem:CFPPretraction} 
CFPP is preserved under retraction.
\end{corollary}

\begin{proof}
Let $P$ be a poset satisfying CFPP. Suppose that $Q$ is an order
retract of $P$ and let $g: Q \rightarrow \mathcal C (Q)$ be an
order preserving map. Let $s: Q \rightarrow P$ and $r:P\rightarrow Q$
be two order preserving maps such that $r\circ s =1_{Q}$ and let
$\overline s: \mathcal C (Q) \rightarrow \mathcal C (P)$ and
$\overline r:\mathcal C (P)\rightarrow \mathcal C (Q)$ be given
by Lemma \ref {lem:C(P)retraction} above.  The map $h: P\rightarrow
\mathcal C (P)$ defined by $h:= \overline s\circ g\circ r $ is
order preserving, hence it has a fixed point, say $x$.  We claim that
$y:=r(x)$ is a fixed point of $g$.  Indeed, since $x\in h(x)$,
$r(x)\in \overline r(h(x))= \overline r\circ \overline s(g(r(x)))=
g(r(x))$, proving our claim.
\end{proof}
 
 \vspace{.3cm}

\begin{definition}
An element $x$ of a poset $P$ is an \emph{irreducible} of $P$ if
either $\{y\in P: y<x\}$ has a largest element or else or $\{y\in P:
y>x\}$ has a least element.
\end{definition}

We note that if $x$ is an irreducible of $P$ then $P_{-x}$, the poset
obtained from $P$ by deleting $x$, is a retract of $P$.

\begin{lemma}\label{lem: joinirreducibility}
Let $P$ be a finite poset and $x$ be an irreducible of $P$.  Then
$\mathcal C (P)$ has FPP if and only if $\mathcal C (P_{-x})$ has
FPP.
\end{lemma}

\begin {proof} 
Since $P_{-x}$ is a retract of $P$, then $\mathcal C (P_{-x})$ is a
retract of $\mathcal C (P)$ (by Lemma
\ref{lem:C(P)retraction}). Since FPP is preserved under retraction, if
$\mathcal C (P)$ has FPP, $\mathcal C (P_{-x})$ has FPP
too. 

Conversely, suppose that $\mathcal C (P_{-x})$ has FPP, and let
$f:\mathcal C (P)\rightarrow \mathcal C (P)$ be an order preserving
map. We prove that $f$ has a fixed point. For that, we will set
$Q:=P_{-x}$, denote by $s$ the identity map from $Q$ to $P$, suppose
that $\{y\in P: y<x\}$ has a largest element $x^-$, and denote by $r$
the retraction map defined on $P$ by $r(x):=x^-$ and $r(y)=y$ for all
$y\not =x$.  Let $\overline s$ and $\overline r$ the maps defined in
Lemma \ref{lem:C(P)retraction} (that is $\overline s(Y):=Conv_P(Y)$
and $\overline r(X):= Conv_Q(r[X])$ if $Y\in \mathcal C (Q)$ and $X\in
\mathcal C (Q)$). The map $g:= \overline r\circ f\circ \overline s$
has a fixed point $Y$. Set $X:= \overline s(Y)$.

\begin{claim}\label{claim: comparability} 
$X\leq f(X)$.  
\end{claim} 


\begin{proof}
Note that $Y\subseteq X\subseteq Y\cup \{x\}$ and $Y\subseteq
f(X)\subseteq Y\cup \{x\}$.
 
\noindent Case 1. $X=Y$, that is $x\not\in X$. 

Subcase 1. $x\not \in f(X)$.  In this case, $\overline r(f(X))=f(X)$
thus $Y=f(X)$ and since $X=Y$, $X=f(X)$ proving our claim.  

Subcase 2. $x\in f(X)$. In this case $f(X)=X\cup \{x\}$ and $x^-\in
X$. Since $x^-\leq x$ this yields $X\leq f(X)$.
 
\noindent Case 2. $x\in X$. In this case $X= Y \cup \{x\}$ and $f(X)=X$. The
first equality amounts to $x\in X$.  For the second note that there
are $a, b\in Y$ such that $a\leq x\leq b$. Since $Y\subseteq f(X)$
this yields $Y\cup \{x\}\subseteq f(X)$. Since $f(X)\subseteq Y\cup
\{x\}$, this gives $f(X)= Y\cup \{x\}$.  
\end{proof}

Now, set $X_0:=X$ and $X_{n+1}:=f(X_{n})$ for every $n\in \N$. We have
$X_n\leq X_{n+1}$. Since $P$ is finite, $\mathcal C (P)$ is finite
too, hence the sequence is stationary. Its largest element is a fixed
point of $f$.
 
Now, suppose that $\{y\in P: y>x\}$ has a least element. Since
$\mathcal C (Q^*)$ is the dual of $\mathcal C (Q)$ it has FPP,
thus the proof above tells us that $\mathcal C (P^*)$ has FPP,
hence by the same token $\mathcal C (P)$ has FPP. \end{proof}

We mention here that we do not know if the the finiteness assumption
in Lemma \ref {lem: joinirreducibility} can be removed.

\begin{definition} 
We recall that a finite poset $P$ is \emph{dismantlable} if there is
an enumeration $x_0, \dots , x_{n-1}$ of its elements such that
$x_{i}$ is irreducible in $P\setminus \{x_j:j<i\}$ for every
$i<n-1$. 
\end{definition}

For example every finite poset with a least, or a largest, element is
dismantlable.

\begin{corollary}\label{cor:dismantlable} 
If a finite poset $P$ is dismantlable then $\mathcal C (P)$ has FPP.
\end{corollary} 

\begin{proof} 
We argue by induction on the cardinality $n$ of $P$. If $n\leq 1$, $P$
has FPP. If $n\geq 2$ then $P$ contains an irreducible element $x$
such that $P_{-x}$ is dismantlable.  By induction $\mathcal C
(P_{-x})$ has FPP. According to Lemma \ref {lem: joinirreducibility}
$\mathcal C (P)$ has FPP.
\end{proof}
 
The following example shows that the converse does not holds.

\begin{example}\label{examplenondismantlable} 
Let $P$ be the ordinal sum of two antichains $\{a,b\}$ and
$\{c,d\}$. Since $P$ has no irreducible it is not dismantlable. On the
other hand $\mathcal C (P)$ has FPP. Indeed, let $f:\mathcal C
(P)\rightarrow \mathcal C (P)$ be an order preserving map. If there is
some $X\in \mathcal C (P)$ such that $f(X)$ is comparable to $X$ then,
since $\mathcal C (P)$ is finite, $f$ has a fixed point. Let $X:=\{c,
d\}$ and $Y:=\{a, b\}$. We may suppose that $f(X)$ is incomparable to
$X$ and $f(Y)$ is incomparable to $Y$. The elements incomparable to
$X$ are $\{c\}$ and $\{d\}$, whereas the elements incomparable to $Y$
are $\{a\}$ and $\{b\}$. With no loss of generality, we may suppose
that $f(X)=$\{c\}$ and $f(Y)=$\{a\}$. We have $Y\leq\{a, c\}\leq X$,
hence $\{a\}\leq f(\{a, c\})\leq
\{c\}$. But then $f(\{a, c\})$ is comparable to $\{a, c\}$. Thus $f$
has a fixed point.
\end{example}

Walker's characterization of finite posets with RFPP leads to:

\begin{theorem}\label{thm:not FPP} 
If a finite poset  $P$ has CFPP then $\mathcal C (P)$ has FPP. 
\end{theorem}

Indeed, if $P$ finite has CFPP then it has RFPP (Theorem
\ref{RFPP}). According to Walker it is dismantlable, thus from
Corollary \ref{cor:dismantlable}, $\mathcal C (P)$ has FPP.

\subsection{Lattices properties and CFPP}
 
Lattice properties do not easily transfer from a poset $P$ to the
poset $\mathcal C(P)$. We give in Example \ref{example:lattice} a
finite lattice $P$ such that $\mathcal C (P)$ is not a lattice (for
an example of infinite and complete $P$ see Example
\ref{example:robert}). Note that since $P$ and $\mathcal C  (P)$
are finite with a largest element, both have CFPP (Fact
\ref{wellfounded}).

\noindent \begin{minipage}[t]{0.5\textwidth} 
\begin{example}\label {example:lattice} 
Let $Q:=\{a, b, 0, 1, c \}\cup \{ij: i<2,j<2\}$ be the $9$-element
poset whose covering pairs are $i<ij$ for $i,j<2$, $a<0$, $a<c$,
$b<c$, $b<1$, and let $P$  be obtained by adding a least and a
largest element to $Q$. Then $P$, shown here, is a lattice. Furthermore, the
subsets $X:= \{ 00, c, 11\}$ and $Y:= \{01, 10\}$ are two convex
subsets (they are antichains of $P$) which have no infimum. Indeed,
let $Z:= (\downarrownogap X)\cap (\downarrownogap Y)= \downarrownogap
\{0, 1\}$.  This is a lower bound of $X$ and $Y$. Inside, $Z_0:= \{0,
a, 1\}$ and $Z_1:=\{0,b, 1\}$ are two maximal lower bounds.
\end{example}
\end{minipage}
\hfill 
\begin{minipage}[t]{0.5\textwidth}
\vspace{1cm}
\begin{center}
\includegraphics[width=8cm]{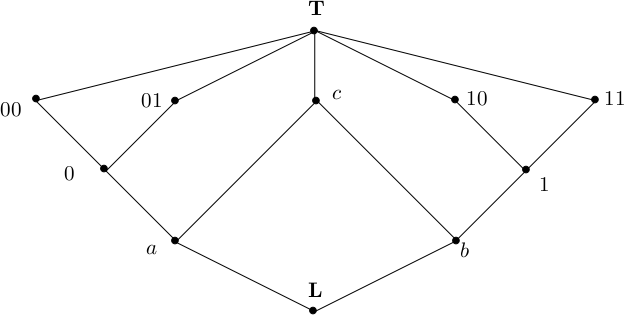}
\end{center}
\bigskip
\end{minipage}

\bigskip

On the other hand, we give in Example \ref{example:robert} below an
example of a complete lattice $T$ without CFPP and such that $\mathcal
C (T)$ does not have $FPP$.  This lattice contains no chain of type
$\theta$.

We recall that \emph{a poset $P$ contains no chain of type $\theta$ if
and only if it is the union of a well founded initial segment and a
dually well founded final segment}. Indeed, set $P^{-}:= \{ x\in P: \; 
 \downarrownogap x\; \text{is well founded} \}$ and $P^{+}:=\{ x\in P: \; 
\uparrownogap x\; \text{is dually well founded}\}$. These two sets
are an initial segment and a final segment respectively. As it is easy
to see, their union is $P$ iff $P$ contains no chain of type $\theta$.
 
\begin{example}\label{example:robert}
Let $2^{<\omega}$, resp. $2^{\omega}$, be the set of finite sequences,
resp. of $\omega$-sequences of $0$ and $1$. Let $s\in 2^{<\omega}\cup
2^{\omega}$; we denote by $dom(s)$ the domain of $s$; the \emph{length}
of $s$, $l(s)$, is the cardinality of $dom(s)$; hence the empty
sequence has length $0$. Let $s, s'\in 2^{<\omega}\cup 2^{\omega}$, we
set $s\leq s'$ if $s=s'_{\restriction dom(s')}$. Set $D:=
2^{<\omega}\times \{0\}$, $U:=2^{<\omega}\times \{1\}$,
$A:=2^{\omega}$.  Let $T:= D\cup U\cup A$ and $S:T\rightarrow T$ define
by $S(x):=x$ if $x\in A$ and $S((s, i)):=(x, i\dot{+}1)$ where the sum
$i\dot{+}1$ is $1$ if $i$ is $0$ and $0$ otherwise. Let $x, y\in T$,
we set $x<y$ in the following cases:

\begin{itemize}
\item $x=(s, i), y=(s', i)$ and  either  $i=0$ and $s<s'$ or $i=1$ and $s'<s$; 
\item $x=(s, 0), y=(s', 1)$ and either $s<s'$ or $s'<s$;
\item $x=(s, 0), y=s'$ and $s<s'$;
\item $x=s, y=(s', 1)$ and $s'<s$.
\end{itemize}
\end{example}

With this in mind we can prove the following claim.

\begin{claim}\label{robertexample}
The set $T$ with the relation $<$  is a complete lattice containing no chain of type
$\theta$. In particular, $\powerset(\omega)$ is not embeddable into
$T$. The lattice $T$ does not have CFPP, the poset $\mathcal C(T)$
is not a lattice and does not have FPP.
\end{claim}

\begin{proof}
The map $S$ is a self dual map which sends $D$ onto $U$, fixes $A$ and
reverses the relation $<$.  The relation $<$ defined on $D\cup A$
yields the binary tree with ends hence yields an ordering and in fact
a meet semilattice. On $A\cup U$, this relation is the dual of $<$,
hence it yields an ordering too which is a join semilattice. To see
that this is an ordering on the union, note that since $D$ is a tree,
$(s, 0)< (s', 1)$ if and only if there is some $t\in A$ such that
$(s,0)<t<(s', 1)$.  Since $T$ is self dual, to prove that this is a
lattice it suffices to prove that this is a join semilattice. Let $x,
y $ be two arbitrary elements of $T$. If $x$ and $y$ are in $A\cup U$
they have a join since $A\cup U$ is a join semilattice. We may suppose
that $x\in D$. If $y$ is comparable to $x$, the join is the largest of
the two. If $y$ is incomparable to $x$ then, as it is easy to see, we
have $x\vee y=x\vee S(y)=S(x)\vee y= S(x)\vee S(y)$. Thus $T$ is a
lattice.  Each maximal chain of $T$ has order type $\omega+1+
\omega^*$ hence is a complete chain. Since $T$ is a lattice, this
ensures that it is complete lattice. The set $D\cup A$ is a well
founded initial segment of $T$ and the set $A\cup U$ is a dually well
founded final segment of $T$; since their union covers $T$, no chain
of type $\theta $ is contained in it.  Let $f: T\rightarrow \mathcal
C(T)$ defined by setting $f(x):= \{(s', 0): l(s')>l(s)\}$ if
$x=(s, 0)$, $f(x):= \{(s', 1): l(s')>l(s)\}$ if $x=(s, 1)$ and $f(x):=
A\setminus \{x\}$ if $x\in A$. This map is order preserving. Since it
does have a fixed point, CFPP fails. 

\medskip

It remains to show that  FPP fails for $\mathcal C (T)$. For
that, we will apply the following result of Rutkowski (\cite{R} Lemma 1):

\begin{fact} \label{fact:rutkowski}
If a poset $P$ contains a totally ordered pregap $( A, B)$ such that
$S( A, B)$ does not have FPP then $P$ does not have FPP.
\end {fact}

To do so, let $\mathcal A:=\{A_n: n\in \N\}$ where $A_n:= \{(s,0)\in D: l(s)\leq
n \}$ and $\mathcal B:=\{B_n: n\in \N\}$ where $B_n:= \{(s,1)\in U:
l(s)\leq n \}$. As it is easy to see, $(\mathcal A, \mathcal B)$ is a
totally ordered pregap. Furthermore, $S(\mathcal A, \mathcal
B)=\{X\subseteq A: X \;\text{is topologically dense in} \; A\}$ (note
that $A$ being the set of branches of the binary tree is homeomorphic
to the Cantor space). Hence, $S(\mathcal A, \mathcal B)$ is an
infinite antichain of $\mathcal C(T)$, thus it does not have
FPP. Fact \ref{fact:rutkowski} now ensures that $\mathcal C (T)$ does
not have FPP. In addition, we get from this construction that
$\mathcal C (T)$ is not a lattice. Indeed, otherwise $S(\mathcal A,
\mathcal B)$ would be a lattice. This is impossible since it is an
infinite antichain.  \end{proof}

\medskip

As a consequence of this last example we obtain the following.

\begin{corollary} \label{cor:failure CFPP}
CFPP is not preserved under finite product. Indeed, CFPP holds for
complete chains, but this property fails for the direct product $[0,
1]\times [0,1]$.
\end{corollary}
 
\begin{proof}
Let $T$ be the lattice defined defined in Example
\ref{example:robert}. This lattice is an order retract of the direct
product $[0, 1]\times [0, 1]$ because it is complete and  can be embedded into this
direct product. Now, CFPP is preserved under order
retraction (Corollary \ref{lem:CFPPretraction}). Thus, if $[0,
1]\times [0, 1]$ had CFPP, $T$ would have CFPP, which is not the
case. 
\end{proof}

\subsection{Selection properties}

\begin{definition}
A poset $P$ has {\em selection property for convex subsets} (CSP
for short) if there is an order preserving map $s :\mathcal
C(P)\rightarrow P$ such that $s(S)\in S$ for every $S\in \mathcal C
(P)$.
\end{definition}

It is a simple exercise to prove that if $P$ has CSP, then $P$ has FPP
if and only if it has CFPP.

CSP is a strong condition, in fact too strong w.r.t. CFPP. Indeed,
CSP is preserved by convex subsets while CFPP is not.   In addition,
note that every finite poset with a least element has CFPP, while 
we will see that the finite lattices in Corollary \ref{crown} do not have CSP. 

\begin{definition}
Let us recall that a poset $P$ is \emph{bipartite} if its
comparability graph $G(P)$ is bipartite; equivalently $P$ is the union
of the set $\min (P)$ of minimal elements and the set $\max (P)$ of
maximal elements. 
\end{definition}

A \emph{crown} is a poset whose comparability graph is a cycle -- a
crown is bipartite and every vertex has degree two.

\begin{lemma} 
A bipartite poset where every vertex has degree at least two does not have CSP.
\end{lemma}

\begin{proof} 
Let $P$ be such a poset. Suppose that there is some selection
$s:\mathcal C (P)\rightarrow P$. Let $P_0:= \min(P)$, $a:=s(P_0)$,
$P_{1}:= \max (P)$, and $b_0:=s(P_1)$. Due to our assumption on $P$, we
have $P_0\leq P_1$ hence $a<b_0$. Since every vertex of $P$ has degree
at least two, there is some $b_1\in P_1$ with $b_1\not = b_0$ and
$a<b_1$. For $i\in \{0, 1\}$ we set $X_i:=(P_0\setminus \{a\}) \cup
\{b_i\}$. As it is easy to check, we have $P_0\leq X_i\leq P_1$ hence
$a\leq s(X_i)\leq b_0$. Since $s(X_i)=b_i$, this yields a
contradiction for $i=1$. \end{proof}

Since CSP is preserved by convex subsets we immediately have:

\begin{corollary} \label{crown}
The lattices made of a crown with a top and bottom element added do  not have CSP. 
\end{corollary}

\begin{definition}
We can weaken CSP by simply supposing that for every chain $\mathcal
C$ in $\mathcal C (P)$ there is an order preserving map $s
:\mathcal{C}\rightarrow P$ such that $s(S)\in S$ for every $S\in
\mathcal{C}$. We will call this property CCSP.
\end{definition}

It is easy to prove that every finite poset satisfies CCSP.  There are
infinite posets with CFPP and not CCSP. We give examples below, but
first recall the notion of {\em cofinality} of a chain.

\begin{definition}
The {\em cofinality} of a chain $C$, written  $cf(\mathcal C)$,  is the least ordinal
$\kappa$ such that $C$ contains a cofinal subset with order type
$\kappa$. 
\end{definition}

Note that the image $C'$ of $C$ by an order preserving map
has either a largest element or the same cofinality than $C$.

In our setting, this yields the following:

\begin{fact}\label{unbounded} 
Let $P$ be a poset with CCSP. Let $\mathcal C$ be a chain in $\mathcal
C (P)$ and $A(\mathcal C):=(\cap_{C\in \mathcal C} \uparrownogap
C)\cap (\cup_{C\in \mathcal C} \downarrownogap C)$. If $A(\mathcal
C)=\emptyset$ then $P$ contains a chain of type $cf(\mathcal C)$.
\end{fact}

\begin{proof} 
Let $\mathcal C'$ be the image of $\mathcal C$ by a selection
map. Since $A(\mathcal C)=\emptyset$, $\mathcal C'$ does not have a
largest element, hence $cf(C')=cf(C)$.
\end{proof}

\vspace{.3cm}

Before we produce the example we need the following lemma. 

\begin{lemma}\label{lem:notCCFP} 
Let $Q$ be a well founded poset such that $cf(Q)=cf(\uparrownogap
x)=\omega_1$ for every $x\in Q$. If $Q$ contains no chain of type
$\omega_1$ then the poset $P:=Q+1$ obtained from $Q$ by adding a
largest element has CFPP and not CCSP.
\end{lemma}

\begin{proof}
Since $P$ is well founded with a largest element, it has CFPP from
Fact \ref{wellfounded}. Next, let $(x_{\alpha})_{\alpha<\omega_1}$ be
a sequence of elements of $Q$ which is cofinal in $Q$. Let
$C_{\alpha}:= Q\setminus \downarrownogap\{ x_{\beta}: \beta <\alpha\}$ for
each $\alpha<\omega_1$. Set $\mathcal C:= \{C_{\alpha}:
\alpha<\omega_1\}$. Let $\alpha\leq \beta<\omega_1$.  The set
$C_{\beta}$ is a final segment of $C_{\alpha }$ and it is cofinal in $Q$.
Indeed, let $x\in Q$, since $cf(\uparrownogap x)= \omega_1$,
$\uparrownogap x\setminus \downarrownogap \{ x_{\gamma}: \gamma
<\beta\}\not= \emptyset$.  Thus $C_{\beta}$ is cofinal in $C_{\alpha}$,
hence, $C_{\alpha}\leq C_{\beta}$ in $\mathcal C (P)$. Consequently,
$\mathcal C$ is a chain. Furthermore $A(\mathcal C)=Q\cap
\bigcap_{\alpha<\omega_1} C_{\alpha }$ hence $A(\mathcal
C)=\emptyset$.  According to Fact \ref{unbounded}, if $P$ had CCSP, it
would contain an uncountable chain.  Hence, CCSP fails.
\end{proof}

\vspace{.3cm}

\begin{definition}
For $\kappa$ be a cardinal, let $[\kappa]^{<\omega}$ be the poset of
finite subsets of $\kappa$ ordered by inclusion.  We further denote by
$[\kappa]^{<\omega}+1$ the complete lattice obtained by simply adding
a largest element to the previous poset.

\end{definition}

\begin{example} 
If a poset $Q$ is up-directed then $cf(Q)=cf(\uparrownogap x)$ for
every $x\in Q$, thus a well founded up-directed poset of cofinality
$\omega_1$ with no chain of type $\omega_1$ will satisfy the
conditions of the Lemma. A natural example is $ [\omega_1]^{<\omega}$.
A second example is $\mathcal I_{<\omega}(Q')$, the set of finitely
generated initial segments of $Q':=\oplus_{\alpha<\omega_1}
L_{\alpha}$, the direct sum of $\aleph_1$ copies of well ordered
chains $L_{\alpha}$ having order type $\alpha$ (the fact that this
poset is well-founded follows from a result of Birkhoff). For an
example of non directed poset, take a regular Aronszajn tree (see
\cite{jech}).\\
\end{example}
 
\begin{remark}
As for CFPP, CSP is not preserved under finite product. Indeed, CSP
holds for chains (see Lemma \ref{prop:CLSPchain}) but this property
fails for the direct product $[0, 1]\times [0,1]$. Otherwise, since
this direct product is a complete lattice, it would have CFPP, which
is not the case according to Lemma \ref{cor:failure CFPP}.\\
\end{remark}

The following result uses a typical example to illustrate the relationship
between fixed  point properties and selection.

\begin{theorem}\label{thm:illus} 
Le $\kappa$ be a cardinal. And let
$P:=[\kappa]^{<\omega}+1$ be the complete lattice defined above.  Then the following hold:\\
[-.2cm]
\begin{enumerate}[{(i)}]
\item  $P$ has  CFPP;\\[-.2cm] 
\item  $P$ has  CSP if and only if $\kappa\leq 2$;\\[-.2cm]
\item  $P$ does not have CCSP if $\kappa$ is uncountable; and,\\[-.2cm]
\item  $\mathcal C (P)$ has FPP if and only if $\kappa<\omega$.\\[-.2cm] 
\end{enumerate} 
\end{theorem}

\begin{proof}
(i). $P$ is well founded with a largest element. Apply Fact \ref{wellfounded}. 

\vspace{.3cm}
\noindent
(ii). If $\kappa\leq 2$ a simple inspection proves that CSP
holds. If $\kappa \geq 3$ then $P$ embeds the lattice $L$ made of the
$6$-element crown with top and bottom added. Since every complete
lattice is a retract of any poset in which it can be embedded, $L$ is
a retract of $P$. Since CSP is preserved under retract, if $P$ had
CSP, $L$ would have CSP. According to Corollary \ref{crown}, this is
not the case.

\vspace{.3cm}
\noindent
(iii). If $\kappa$ is uncountable, a surjective map from $\kappa$
onto $\omega_1$ induces a retraction of $P$ onto $[\omega_1]^{<\omega}+1$. Since
CCSP is preserved under retract, if $P$ had CCSP, $[\omega_1]^{<\omega}+1$ would
have CCSP. According to Lemma \ref{lem:notCCFP}, this is not the case.

\vspace{.3cm}
\noindent
(iv). If $\kappa<\omega$, $P$ is finite.  Since it has a largest 
element, $\mathcal C (P)$ has FPP  by Lemma \ref{lem: wqo+C(P)}.  If
$\kappa\geq \omega$, the poset $Q$ defined in Example
\ref{examplenotFPP} is embeddable in $[\kappa]^{<\omega}$ because for
every $x\in Q$, $\downarrownogap x$ is finite. Thus $\overline Q$ is
embeddable in $P$. Since $\overline Q$ is a complete lattice, this is
a retract of $P$. According to Lemma \ref{lem:C(P)retraction},
$\mathcal C (\overline Q)$ is a retract of $\mathcal C (P)$. Since FPP
is preserved under retraction, and $\mathcal C (\overline Q)$ does not have FPP,
$\mathcal C (P)$ does not have FPP. \end{proof}\\

\section{The lattice of convex sublattices of a lattice}\label{sublattices}

Let $T$ be a lattice. The join and meet of two elements $x,y\in T$
are, as usual, denoted respectively by $x\vee y$ and $x\wedge y$ and a
sublattice is a nonempty subset closed under these operations.  The
set $\mathcal {I}d (T)$ of ideals of $T$, ordered by inclusion, is a
lattice (a complete lattice provided that $T$ has a least element),
the join, and meet, of two ideals $A$ and $B$ being $$A\vee B \ = \
\downarrownogap\{a \vee b: a\in A, b\in B\} \ \text{and} \ A\wedge B =
A\cap B \ .$$

\noindent
Similarly, $\mathcal {F}i(P)$, the set of filters of $T$ is a lattice
(a complete lattice provided that $T$ has a largest element): for
every $A, B\in \mathcal {F}i (T)$ $$A\vee B \ = \ \uparrownogap\{a
\wedge b: a\in A, b\in B\} \ \text{and} \ A\wedge B= A\cap B. $$

It is it immediate that an ideal of a lattice is a nonempty initial
segment closed under pairwise joins, and a filter of a lattice is a
final segment closed under pairwise meets.  Hence, the up and down
directed convex subsets of $T$ are simply the convex sublattices of
$T$.  

We denote by $\mathcal C_L(T)$ the set of nonempty convex sublattices
of $T$, ordered with the bi-domination preorder.

\begin{proposition} Let  $T$ be a lattice. Then: 
\begin{enumerate} [{(a)}]

\item  $\mathcal{C}_{ L}(T)$ is a lattice and the map  
$\vartheta: \mathcal{C}_{ L}(T)\rightarrow \mathcal {I}d(T) \times
\mathcal {F}i(T)^*$, defined by $\vartheta(S):= (\downarrownogap S,
\uparrownogap S)$ for $S\in \mathcal{C}_{ L}(T)$, is a one to one
lattice homomorphism. The image is the subset $$\mathcal K(T):= \{(I,
F)\in {I}d(T)\times \mathcal {F}i(T)^*: I\cap F\not =\emptyset\}$$ of
$\mathcal {I}d(T)\times \mathcal {F}i(T)^*$.\\

\item In particular, if  $A, B\in \mathcal C_{L}(T)$,  then:\\[-.2cm]
\begin{enumerate}[{(i)}]

\item $A\leq B\; \text{if and only if} \ a\vee b \in B \ \text
{and} \ a\wedge b \in A \ \text {for every} \ a\in A, b\in B$;
\\[-.2cm]
    
\item $A\vee B=( (\downarrownogap A) \vee (\downarrownogap B)) \cap(( \uparrownogap A ) \cap (\uparrownogap B)) ; \ \text{and}$  \\[-.2cm]
    
\item $A\wedge B=( (\downarrownogap A) \cap (\downarrownogap B))
\cap(( \uparrownogap A ) \wedge(\uparrownogap B)). $
   
\end{enumerate}
 
\end{enumerate}
 
\end{proposition}

\begin{proof}
{\it (a).}  From Fact \ref {fact3.1} the map $\vartheta$ is an order
isomorphism from $\mathcal{C}_{ L}(T)$ on its image. Thus, to prove
that $\mathcal{C}_{ L}(T)$ is a lattice, it suffices to prove that its
image is a lattice. This image is clearly included in $\mathcal
K(T)$. The reverse inclusion is due to the fact that if $(I, F)\in
\mathcal K(T)$ then $\downarrownogap(I \cap F)=I$ and
$\uparrownogap(I\cap F)=F$.  In fact, the first equality holds
whenever $I$ is an ideal and $F$ is a final segment which meets $I$.
With the product order, $\mathcal {I}d(T)\times \mathcal {F}i(T)^*$ is
a lattice in which $(I,F)\vee(I',F')=(I\vee I', F\cap F')$ for every
$(I,F), (I',F')\in \mathcal {I}d(T)\times \mathcal {F}i(T)^*$.  If
$(I,F), (I',F')\in \mathcal K(T)$ then $(I,F)\vee(I',F')\in \mathcal
K(T)$ since $$\{x\vee x': x\in I\cap F, x'\in I'\cap F'\} \subseteq
(I\vee I')\cap (F\cap F').$$ The same property holds with meet instead
of join. Hence $\mathcal K(T)$ is a sublattice of the lattice
$\mathcal {I}d(T)\times \mathcal {F}i(T)^*$.

\noindent
{\it (b).}  This follows from {\it (a)} and the form of joins and
meets in $\mathcal {I}d(T)\times \mathcal {F}i(T)^*$. \end{proof}

\begin{lemma}\label{lem:filling}
Let $(\mathcal A,  \mathcal B)$ be a pregap of $\mathcal{C}_{ L}(T)$. Then the following properties are equivalent:
\begin{enumerate}[{(a)}]
\item $(\mathcal A, \mathcal B)$ is a gap of $\mathcal{C}_{ L}(T)$;
\item  $I_{\mathcal B}\cap F_{\mathcal A} = \emptyset$;
\item  $(\mathcal A, \mathcal B)$ is a gap of $\mathcal{C}(T)$.
\end{enumerate}

\end{lemma}  
\begin{proof} $(b) \Rightarrow (c)$:  This is the contrapositive of  {\it(i)} of Lemma \ref{lem:pregaps}. 

\noindent
$(c) \Rightarrow (a)$: This follows from the fact that $\mathcal{C}_{ L}(T)$ is a subposet of $\mathcal{C}(T)$. 

\noindent
$(a) \Rightarrow (b)$: Suppose that $(b)$ does not hold. Then observe
that $I_{\mathcal B}$ and $F_{\mathcal A}$ are up and down directed,
respectively. Thus Lemma \ref{lem:pregapsbis} applies and from
implication $(c)\Rightarrow (b)$ of $(ii)$ of this Lemma, $(\mathcal
A, \mathcal B)$ is not a gap in $\mathcal{C}(T)$. According to Lemma
\ref{lem:pregaps}{\it (i)}, $I_{\mathcal B}\cap F_{\mathcal A}$
separates $(\mathcal A, \mathcal B)$ in $\mathcal{C}(T)$. But
$I_{\mathcal B}\cap F_{\mathcal A}\in \mathcal{C}_{ L}(T)$, hence it
separates $(\mathcal A, \mathcal B)$ in $\mathcal{C}_{ L}(T)$, hence
$(a)$ does not hold.
\end{proof}

The proof of the following lemma is straightforward. 

\begin{lemma} \label {lem:wood2}
Let $T$ be a lattice, the map $i:T\rightarrow \mathcal{C}_{ L}(T)$
defined by $i(x):=\{x\}$ is a lattice homomorphism, it preserves all
infinite joins and meets in $T$ and all gaps of $T$.
\end{lemma}

Now since a complete lattice has no gap, we deduce immediately that:

\begin{corollary}\label{cor:wood2}
If $\mathcal{C}_{ L}(T)$ is a complete lattice then $T$ is a complete
lattice.
\end{corollary}

\begin{lemma}\label{lem:embeddingtest}
Let $T$ be a complete lattice. There is an order embedding from
$\powerset(\omega)$ into $T$ if and only if there are sequences
$(x_n)_{n \in \omega}$ and $(y_n)_{n \in \omega }$ in $T$ such that:

\begin{enumerate}[{(i)}]
\item  $x_0>x_1> x_2>  \dots x_n> \dots  $;\\[-.2cm]
\item  $y_0\not \leq x_0, \ y_1\not \leq x_1\vee y_0, \ y_2\not \leq x_2\vee y_0\vee y_1, \  \dots \ , \ y_n\not\leq x_n\vee \bigvee_{j<n}y_j, \dots.$;\\[-.2cm]
\item $y_1\leq x_0, \ y_2\leq x_1, \ \dots \ , \ y_{n+1}\leq x_n, \dots$.
\end{enumerate}
\end{lemma}

\begin{proof} 
Suppose that $f$ is an order embedding from $\powerset(\omega)$ into
$T$.  Set $x_n:=f(\omega\setminus \{j: j\leq n\})$ and $y_n:=f(\{n\})$
for each $n\in \omega$.  Then $(i)$ and $(iii)$ are trivially
satisfied. Concerning $(ii)$, note first that since $\{n\}\not
\subseteq \omega \setminus \{n\}$ we have $y_n=f(\{n\})\not \leq
f(\omega\setminus \{n\})$ and next, that since $\omega\setminus \{j:
j\leq n\}\cup \bigcup_{j<n} \{j\} \subseteq \omega \setminus \{n\} , $
we have $$ x_n\vee \bigvee_{j<n}y_j =f(\omega\setminus \{j: j\leq
n\})\vee\bigvee_{j<n}f(\{j \}) \leq f(\omega\setminus \{n\}).$$ Thus
$y_n\not\leq x_n\vee \bigvee_{j<n}y_j$ as required.

Conversely, suppose that there are two sequences satisfying $(i) -
(iii)$. Define $f:\powerset(\omega)\rightarrow T$ by setting $f(X):=
\bigvee\{y_n: n\in X\}$ for every $X\subseteq \omega $, with the
convention that $f(\emptyset)$ is the least element of $T$. Clearly,
this map is order preserving. To show that it is an order embedding,
it suffices to prove that if $X\not \subseteq X'$ then $f(X)\not \leq
f(X')$. Suppose that $n\in X\setminus X'$. Then $X'\subseteq \omega
\setminus \{n\}$. Since $f$ is order preserving, it follows that
$f(X') \leq f(\omega \setminus \{n\})$. By definition of $f$, we
have 
\[ f(\omega\setminus \{n\})= \bigvee_{j> n}f(\{j\} )\vee
\bigvee_{j<n} f(\{j\})= \bigvee_{j>n}y_{j}\vee \bigvee_{j<n} y_j\leq
x_n\vee \bigvee_{j<n}y_j\]
since $y_j\leq x_n$ for $j>n$.  Thus,
$f(X')\leq x_n\vee \bigvee_{j<n}y_j$. Since $n\in X$, we have $y_n=
f(\{n\}) \leq f(X)$. Since $y_n\not \leq x_n\vee \bigvee_{j<n}y_j$,
this yields $f(X)\not \leq f(X')$, as required.  \end{proof}

\medskip

We note that the conditions in the preceding lemma were introduced in
\cite{P-R} to study pregaps under lattice homomorphisms.

\begin{lemma}\label{lem:main} 
Let $T$ be a complete lattice. If $\mathcal{C}_{ L}(T)$ is not complete then there is an embedding from $\powerset(\omega)$ into $T$. 
\end{lemma}

\begin{proof} Suppose that  $\mathcal{C}_{ L}(T)$ is not complete. We begin by the following claim. 

\begin {claim} \label {claim:specialgap} 
$\mathcal{C}_{ L}(T)$ contains a special gap $(\mathcal A,\mathcal B)$
where $\mathcal A$ and $\mathcal B$ are, respectively, up-directed
with a least element $A_0$ and down-directed with a largest element
$B_0$. 
\end{claim}

\noindent
{\bf Proof of Claim \ref{claim:specialgap}.} First $\mathcal{C}_{
L}(T)$ contains a gap, not necessarily special, with these
properties. Indeed, since $\mathcal{C}_{ L}(T)$ is not complete, it
contains a gap, say $(\mathcal A,\mathcal B)$. Since $C_{L}(T)$ has a
least element and a largest element (namely $\{0_T\}$ and $\{1_T\}$,
where $0_T$ and $1_T$ are the least and largest elements of $T$),
$\mathcal A$ and $\mathcal B$ are both nonempty. Since $\mathcal{C}_{
L}(T)$ is a lattice, we may suppose that $\mathcal A$ is up-directed
and $\mathcal B$ down-directed (otherwise, replace $\mathcal A$ by
$L(\mathcal B)$ and $\mathcal B$ by $U(L(\mathcal B))$) and also that
$\mathcal A$ and $\mathcal B$ have, respectively, a least element
$A_0$ and a largest element $B_0$.  The pair $(\mathcal A_{\mathcal
B}, \mathcal B_{\mathcal A})$ is a special gap with the required
properties. Indeed, according to Lemma \ref{lem:filling}, $(\mathcal
A,\mathcal B)$ is a gap in $C (T)$. Furthermore $I_{\mathcal B}$ and
$F_{\mathcal A}$ are, respectively, up- and down-directed. Hence, by
Lemma \ref{lem:pregapsbis}, $(\mathcal A_{\mathcal B}, \mathcal
B_{\mathcal A})$ is a gap in $C (T)$, and thus in $\mathcal{C}_{
L}(T)$, and $\mathcal A_{\mathcal B}$ and $\mathcal B_{\mathcal A}$
are, respectively, up-directed with a least element and down-directed
with a largest element. \endproof \\

\begin{claim} \label{claim:consequencenewgap} 
If $Y$ and $X$ are two subsets of $T$ such that $Y\cap A$ and
$X\cap B$ are nonempty for all $A \in \mathcal A$ and
$B  \in \mathcal B$ then $\bigvee Y\not\leq\bigwedge X$.
\end{claim}

\noindent
{\bf Proof of Claim \ref{claim:consequencenewgap}.}  Suppose that the
conclusion does not hold. Let $d$ satisfy $\bigvee Y\leq d\leq
\bigwedge X$. Since $d\leq \bigwedge X$, $d\in \bigcap _{B \in
\mathcal B}\downarrownogap B=I_{\mathcal B}$ Similarly, since $\bigvee
Y\leq d$, $d\in \bigcap _{A \in \mathcal A}\uparrownogap A=F_{\mathcal
A}$. Hence $d\in I_{\mathcal B}\cap F_{\mathcal A}$. This is
impossible as $I_{\mathcal B}\cap F_{\mathcal A} =
\emptyset$.
\endproof

Now the following claim provides the constructions of the desired elements. 

\begin{claim} \label{claim:sequenceelements}
There are $(x_n : n \in \omega ) \subseteq B_{0}$ and and $(y_n : n
\in \omega ) \subseteq A_{0}$ satisfying conditions $(i) - (iii)$ of
Lemma \ref{lem:embeddingtest}.
\end{claim}

\noindent
{\bf Proof of Claim \ref{claim:sequenceelements}.}  Let $n\in \omega$
and suppose $x_k$ and $y_k$ have been defined for all $k<n$. Define $x_n$ and
$y_n$ as follows.  If $n=0$, set $Y_0:= A_{0}$ and $X_0:= B_{0}$. The
hypotheses of Claim \ref{claim:consequencenewgap} are satisfied, hence
there are $y_0\in Y_{0}$ and $x_0\in X_{0}$ such that $y_0\not \leq
x_0$. If $n>0$, set
\begin{align*}
Y_n &:= \ \downarrownogap x_{n-1}\cap A_{0}, \ Z_n := \ \downarrownogap  x_{n-1}\cap B_{0}, \ \text{and} \\
X_n &:= \ Z_n\vee  \bigvee_{k<n} y_{k} ( = \{x\vee \bigvee_{k<n} y_{k}: x\in Z_n\}).
\end{align*}
The sets $X_{n}$ and $Y_{n}$ satisfy the hypotheses of Claim
\ref{claim:consequencenewgap}.  Indeed, let $A\in \mathcal A$ and
$B\in \mathcal B$. Since $A\leq B_{0}$ and $x_{n-1}\in B_{0}$, there
is some $t\in A$ with $t\leq x_{n-1}$. Since $A\subseteq A_0$, $t\in
Y_n\cap A$ proving that $Y_n\cap A \ne \emptyset$.  Since $B\leq
B_{0}$ and $x_{n-1}\in B_{0}$ there is some $x\in B$ such that $x\leq
x_{n-1}$, and since $B \subseteq B_0$, $x\in B_0$ and, thus, $x\in Z_n$.
Also, since $A_{0}\leq B$ and $\bigvee_{k<n} y_{k}\in A_{0}$, $x\vee
\bigvee_{k<n} y_{k}\in B$, hence $X_n\cap B\not = \emptyset$.  From
Claim \ref{claim:consequencenewgap} there are $y_n\in Y_{n}$ and
$x_n\in X_{n}$ such that $y_n\not \leq x_n\vee \bigvee_{k<n} y_{k}$.
With the fact that necessarily $x_n<x_{n-1}$, all conditions of Lemma
\ref{lem:embeddingtest} are satisfied. \endproof \\

Finally,  according to Lemma \ref{lem:embeddingtest},  there is an embedding from
$\powerset (\omega)$ into $T$ and this completes the proof. 
\end{proof}

\begin{lemma} \label {lem:wood1}
Let $P$ and $Q$ be lattices. If $Q$ is an order retract of $P$ or a
lattice quotient of $P$ then $\mathcal{C}_{ L}(Q)$ is an order retract
of $\mathcal{C}_{ L}(P)$.
\end{lemma}

\begin{proof}
Suppose that $Q$ is an order retract of $P$. Let $s: Q \rightarrow P$
and $r:P\rightarrow Q$ be two order preserving maps such that $r\circ
s =1_{Q}$ and let $\overline s: \mathcal C (Q) \rightarrow \mathcal C
(P)$ and $\overline r:\mathcal C (P)\rightarrow \mathcal C (Q)$ be
given by Lemma \ref {lem:C(P)retraction}.  These two maps send
$\mathcal C_L (Q)$ into $\mathcal C_L (P)$ and $\mathcal C_L (P)$ onto
$\mathcal C_L (Q)$.  Thus, we have coretraction and retraction maps
and $\mathcal{C}_{ L}(Q)$ is an order retract of $\mathcal{C}_{
L}(P)$.

Suppose that $Q$ is a quotient of $P$. Let $r: P\rightarrow Q$ be a
surjective lattice homomorphism.  Let $\overline s:
\mathcal{C}_L(Q)\rightarrow \mathcal{C}_{ L}(P)$ and $\overline r:
\mathcal{C}_{ L}( P)\rightarrow \mathcal{C}_{ L}(Q)$ be defined by
setting $\overline s(Y):=r^{-1}(Y)$ and $\overline r(X):=Conv_Q(r[X])$
for all $Y\in \mathcal{C}_{ L}(Q)$ and $X\in \mathcal C_L (P)$. These
maps are order preserving.  Moreover, if $Y\in \mathcal{C}_{ L}(Q)$
then, since $r[r^{-1}(Y)]=Y$, $\overline r\circ \overline s(Y)=Y$.
Hence, $\overline s$ and $\overline r$ are coretraction and retraction
maps; in particular $\mathcal{C}_{ L}(Q)$ is an order retract of
$\mathcal{C}_{ L}(P)$.
\end{proof}

Since a retract of a complete lattice is complete, Lemma
\ref{lem:wood1} yields immediately:
 
\begin{corollary}\label{cor:CL_{*}(T)retract+quotient} 
If $Q$ is a lattice quotient of a retract of $P$ and $\mathcal{C}_{
L}(P)$ is complete, then $\mathcal{C}_{ L}(Q)$ is complete too.
\end{corollary}

Another consequence is this:

\begin{corollary} \label{cor:wood1} 
If $Q$ is a lattice quotient of a complete lattice $P$ and $P$ is
order embeddable in a lattice $T$ then $\mathcal{C}_{ L}(Q)$ is a
retract of $\mathcal{C}_{ L}(T)$.
\end{corollary}

\begin{proof}  
Since every complete lattice which is embeddable in a lattice is an
order retract of that lattice, $P$ is an order retract of $T$. From
Lemma \ref{lem:wood1}, $\mathcal{C}_{ L}(P)$ is a retract of
$\mathcal{C}_{ L}(T)$ and $\mathcal{C}_{ L}(Q)$ is a retract of
$\mathcal{C}_{ L}(P)$.  Thus $\mathcal{C}_{ L}(Q)$ is a retract of
$\mathcal{C}_{ L}(T)$ as claimed. \end{proof}\\

For any set $E$,  let $\powerset(E)$ denote the set of all subsets of $E$ ordered by containment 
and let $\powerset(E)/Fin$ be the quotient of $\powerset(E)$ by the ideal $Fin$ of finite subsets of 
$E$. Define $p:\powerset(E)\rightarrow \powerset(E)/Fin$ to be the canonical projection. For $X,
Y \in \powerset(E)$, we set $X\leq_{Fin} Y$ if $X\setminus Y\in Fin$.
This defines a quasi-order on $\powerset(E)$, its image under $p$ is the
order on $\powerset(E)/Fin$.

With all these tools in hand, we are ready to provide the proof of Main Theorem~\ref{thm:main}.

\subsection{Proof of Main Theorem \ref{thm:main}}

$(i) \Rightarrow (ii):$ Let $Q$ be a quotient of a retract of $T$.
According to Corollary \ref{cor:wood1}, $\mathcal{C}_{ L}(Q)$ is a
retract of $\mathcal{C}_{ L}(T)$. Since $\mathcal{C}_{ L}(T)$ is
complete, $\mathcal{C}_{ L}(Q)$ is complete too. According to
Corollary \ref{cor:wood2}, $Q$ complete.

\medskip

\noindent $(ii) \Rightarrow (iii):$ Since $T$ is a quotient and a retract of itself, it is complete. Let $P:=\powerset(\omega)$ and $Q:= \powerset(\omega)/Fin$. Clearly, the lattice $Q$ is not complete. Thus $P$ cannot be a retract of $T$. \\[-.2cm]

\medskip 

\noindent $(iii) \Rightarrow (i):$ Apply Corollary \ref{cor:wood2} and Lemma \ref{lem:main}. \endproof

\medskip

\subsection{Selection property, fixed point property and completeness of the lattice of convex sublattices}

\begin{proposition}\label{lem:CLFPPretraction}  
CLFPP is preserved under retraction and lattice quotient.
\end{proposition}

\begin{proof} 
Let $P$ be a lattice.  Let $Q$ be an order retract of $P$ or a lattice
quotient of $P$. According to Lemma \ref{lem:wood1}, $\mathcal{C}_{
L}(Q)$ is an order retract of $\mathcal{C}_{ L}(P)$. More precisely,
if $Q$ is an order retract, let $s: Q \rightarrow P$ and
$r:P\rightarrow Q$ be two order preserving maps such that $r\circ s
=1_{Q}$ and let $\overline s: \mathcal C(Q)
\rightarrow \mathcal C(P)$ and $\overline r:\mathcal C(P)\rightarrow
\mathcal C(Q)$ be given by Lemma \ref {lem:C(P)retraction}.  These two
maps provide a retraction of $\mathcal{C}_L(P)$ onto $\mathcal{C}_{
L}(Q)$.

If $Q$ is a lattice quotient of $P$ let $r: P\rightarrow Q$ be a
surjective lattice homomorphism. Let $\overline s: \mathcal{C}_{
L}(Q)\rightarrow \mathcal{C}_{ L}(P)$ and $\overline r: \mathcal{C}_{
L}( P)\rightarrow \mathcal{C}_{ L}(Q)$ be defined by setting
$$\overline s(Y):=r^{-1}(Y) \ \text{and} \ \overline
r(X):=Conv_Q(r[X])$$ for all $Y\in \mathcal{C}_{ L}(Q)$ and $X\in
\mathcal{C}_L (P)$. According to Lemma \ref{lem:wood1}, $\overline s$
and $\overline r$ are coretraction and retraction maps. Thus, in both
cases, if $g: Q \rightarrow \mathcal{C}_{ L}(Q)$ is an order
preserving map, we may proceed as in the proof of Corollary
\ref{lem:CFPPretraction}, defining $h: P\rightarrow \mathcal{C}_{
L}(P)$ by $h:= \overline s\circ g\circ r $.  Then, if $x$ is a fixed
point of $h$, $y:=r(x)$ is a fixed point of $g$.  \end{proof}\\
 
\begin{corollary}\label {failureCLFPP} 
The lattice $\powerset (\omega)$ does not have CLFPP.
\end{corollary}

\begin{proof}
The lattice $\powerset (\omega)/Fin$ is a lattice quotient of
$\powerset (\omega)$.  Since it is not complete, it does not have FPP
\cite{davis}. In particular, it does not have CLFPP.  According to
Lemma \ref{lem:CLFPPretraction}, $\powerset (\omega)$ cannot have
CLFPP. \end{proof}

\begin{definition}
Recall that a lattice $T$ has the \emph{selection property for convex
sublattices} (CLSP for short) if there is an order preserving map
$\varphi: \mathcal{C}_{ L}(T)\rightarrow T$ such that $\varphi(S)\in
S$ for every $S\in \mathcal{C}_{ L}(T)$.  
\end{definition}

\begin{proposition}\label{routineproposition}
Let us consider the following properties of a lattice 
$T$: 

\begin{enumerate}[{(i)}]
 \item $T$ is complete and has  CLSP;\\[-.3cm]
 \item $T$ has  CLFPP;\\[-.3cm]
 \item $\mathcal{C}_{ L}(T)$ is a complete lattice. \\[-.3cm]
\end{enumerate}

\noindent
Then $(i)\Rightarrow(ii)\Rightarrow (iii)$. 
\end{proposition}

\begin{proof}
$(i)\Rightarrow (ii)$. Let $h: T \rightarrow \mathcal{C}_{ L}(T)$ be
an order preserving map.  Let $\varphi:\mathcal{C}_{ L}(T)\rightarrow
T$ be an order preserving selection map guaranteed by $(i)$.  The map
$\varphi\circ h: T \rightarrow T$ is order preserving. Since $T$ is a
complete lattice, $\varphi \circ h$ has a fixed point \cite{tarski},
say, $x=\varphi \circ h(x)$.  Since $\varphi$ is a selection, $x\in
h(x)$. Hence, $(ii)$ holds.

\medskip

\noindent $(ii)\Rightarrow (iii)$.  Suppose that $(ii)$ holds. According to Lemma   \ref{lem:CLFPPretraction} and 
Corollary \ref{failureCLFPP}, $\powerset (\omega)$ is not a retract of
$T$. Since $\powerset (\omega)$ is a complete lattice, it cannot be
embedded in $T$. Hence, according to implication $(iii)\Rightarrow
(i)$ of Main Theorem \ref{thm:main}, $\mathcal{C}_{ L}(T)$ is a complete
lattice. \end{proof}

\begin{remark} \label{rem: shorter} Here is a more direct route to the implication $(ii)\Rightarrow (iii)$.
As in the proof of Lemma \ref{lem:main}, observe that in the non complete lattice $\mathcal{C}_{ L}(T)$ 
there is a totally ordered gap $(\mathcal A,\mathcal B)$.  According to Lemma \ref{lem:filling},  
$I_{\mathcal B}\cap F_{\mathcal A} = \emptyset$. This gap is then a gap in $\mathcal C(T)$.  
Lemma \ref{lem:fix-pointfree} yields a fixed point free map of $T$ to $\mathcal{C}_{ L}(T)$.\\
\end{remark}

Trivially, we have:

\begin{proposition}
The dual of a lattice with  $CLSP$ has  CLSP.
\end{proposition}

\begin{proposition} 
Given lattices $L_0$ and $L_1$, let $L:=L_0\times L_1$ be their direct
product and $\pi_i:L\rightarrow L_i$ be the $i$-th projection for
$i<2$. Then the map $\pi: \mathcal{C}_{ L}(L)\rightarrow \mathcal{C}_{
L}(L_0)\times \mathcal{C}_{ L}(L_1)$ defined by $\pi(S):=( \pi_0[S],
\pi_1[S])$ for all $S\in \mathcal{C}_{ L}(L)$ is an isomorphism from
$\mathcal{C}_{ L}(L)$ onto $\mathcal{C}_{ L}(L_0)\times \mathcal{C}_{
L}(L_1)$.
\end{proposition}

\begin{proof} 
Let $\pi':\mathcal{C}_{L}(L_0)\times \mathcal{C}_{ L}(L_1)\rightarrow
\mathcal{C}_{ L}(L)$ be defined by $\pi'(S_0, S_1):= S_0\times
S_1$. Observe that $\pi'$ is the inverse of $\pi$. \end{proof}

\begin{corollary} \label{cor:finiteproduct}
The class of lattices $T$ such that $\mathcal{C}_{ L}(T)$ is complete
is closed under finite product.\\
\end{corollary}

\begin{proposition} \label{lem:CLSPproduct} 
$CLSP$ is preserved under finite products.  
\end{proposition}

\begin{proof}
Let $L_0$ and $L_1$ be lattices with CLSP. Let $L:=L_0\times L_1$, let
$\pi_i$ be the $i$-th projection, and let $\varphi_i: \mathcal{C}_{
L}(L_i) \rightarrow L_i$ be an ordered preserving selection map for
$i<2$.  Let $\varphi: \mathcal{C}_{ L}(L)\rightarrow L$ be defined by
$\varphi(S):= (\varphi_0(\pi_0[S]), \varphi_1(\pi_1[S]))$ for $S\in
\mathcal{C}_{ L}(L)$. Since $S=\pi_0[S]\times \pi_1[S]$ for all $S\in
\mathcal{C}_{ L}(L)$, $\varphi$ is an order preserving selection map.
\end{proof}

\begin{proposition} \label{lem:quotient-retract} 
Every quotient of a lattice with $CLSP$ is a retract of that
lattice. \end{proposition}

\begin{proof} 
Let $Q$ be a quotient of a lattice $P$ with CLSP with a surjective
homomorphism $q: P\rightarrow Q$ .  For each $y\in Q$, the set
$q^{-1}(y)$ belongs to $\mathcal{C}_{ L}(P)$.  Moreover, the map
$q^{-1}: Q\rightarrow \mathcal{C}_{ L}(P)$ is order preserving. Let
$s:\mathcal C_{L}(P)\rightarrow P$ be an order preserving
selection. Let $f:=s\circ q^{-1}$. Clearly $q\circ f= 1_{Q}$. Hence,
$q$ is a retraction and $f$ a coretraction. Thus $Q$ is a retract of
$P$.  \end{proof}\\

\begin{proposition} \label{lem:CLSPretraction}
$CLSP$ is preserved under retraction.
\end{proposition}

\begin{proof} 
Let $P$ be a lattice with CLSP and let $Q$ be an order retract of
$P$.  Let $s: Q\rightarrow P$ and $r: P\rightarrow Q$ be order
preserving maps such that $r\circ s =1_{Q}$. Let $\overline s:
\mathcal{C}_L (Q)\rightarrow \mathcal{C}_L (P)$ and $\overline r:
\mathcal{C}_L( P)\rightarrow \mathcal{C}_ L(Q)$ be defined as in the
proof of Lemma \ref{lem:wood1}.  Let $\varphi: \mathcal
C_{L}(P)\rightarrow P$ be an order preserving selection.  Let $\psi:
\mathcal{C}_L (Q)\rightarrow Q$ be defined by $\psi:= \overline r \circ
\varphi \circ \overline s$. Clearly, this map is order preserving. We
claim that this is a selection map. Indeed, let $Y\in \mathcal{C}_{
L}(Q)$.  Since $\varphi(\overline s (Y))\in \overline s(Y)$ it follows
that $\psi (Y)=  \overline r (\varphi(\overline s(Y))) \in \overline r \circ
\overline s(Y)$. Since, according to Lemma \ref {lem:wood1},
$\overline r\circ \overline s(Y)=Y$, it follows that $\psi(Y) \in
Y$. \end{proof}

Since a complete lattice which is embeddable in a poset is a retract
of that poset, Proposition \ref{lem:CLSPretraction} immediately
yields:

\begin{corollary} \label{cor:retract} 
Every complete lattice which is embeddable in a lattice with $CLSP$
has $CLSP$.\\
\end{corollary}

Combining Propositions \ref{lem:quotient-retract} and \ref{lem:CLSPretraction}, we get:

\begin{corollary}\label{cor:CLSPquotient}  
$CLSP$ is preserved under lattice quotients.
\end{corollary}

Several examples of lattices with $CLSP$ can be obtained from the
following result of \cite{D-R-S}.  For the reader's convenience, we
recall the proof.

\begin{proposition} \label{prop:CLSPchain} 
Every chain has $CLSP$. \end{proposition}

\begin{proof} 
Let $C:= (E,\leq)$ be a chain and let $\leq_{wo}$ be a well ordering
on $E$.  Define $\varphi: \mathcal{C}_{ L}(C)\rightarrow C$ by setting
$\varphi(S)$ to be the least element of $S \in \mathcal{C}_L (C)$ with
respect to the well-ordering $\leq_{wo}$. This map is order
preserving.  Indeed, let $S', S''\in \mathcal C_{L}(C)$ such that $S'
\leq S''$.  Let $x':= \varphi(S')$ and $x'':= \varphi(S'')$.  If
$x',x''\in S'\cap S''$ then $x' \leq_{wo} x''$ and $x''\leq_{wo} x'$,
thus $x'=x''$. If $x'\in S'\setminus (S'\cap S'')$ then, since $S'\leq
S''$, $S'\cap S''$ is a final segment of $S''$. Thus $x'< x''$. If
$x''\in S''\setminus (S'\cap S'')$, a similar argument yields
$x'<x''$. From this, $\varphi$ is an order preserving
selection. \end{proof}

Combining Propositions \ref{lem:CLSPproduct},
\ref{lem:CLSPretraction}, and \ref{prop:CLSPchain}, we get:

\begin{proposition}
Every retract of a product of finitely many chains has $CLSP$.
\end{proposition}

\begin{corollary}\label{prop:finitedim} 
Every complete finite dimensional lattice has $CLSP$.
\end{corollary}

\begin{proof} 
A poset $P$ of dimension $n$ embeds in a direct product $L$ of $n$
chains (and not fewer).  If $P$ is complete then it is a retract of
$L$, thus from Proposition \ref{prop:finitedim} it has $CLSP$.
\end{proof}

Countable lattices have the $CLSP$ as well.  One proof follows directly
from Proposition \ref{lem:CLSPretraction} and the following two
facts.

\begin{proposition}(Proposition 7 \cite{pou-riv84}). 
A sublattice of the lattice of all finite unions of intervals of some
chain, ordered by containment, has the $CLSP$.
\end{proposition}

\begin{proposition}(Corollary 3 \cite {pou-riv84}) 
Every countable lattice is a retract of a convex sublattice of a
Boolean algebra generated by a countable chain.
\end{proposition}

We also provide a direct argument based on a variant of the lattice
``weaving argument" (see \cite{dean}).

\begin{theorem} Every  countable lattice has  $CLSP$.
\end{theorem}

\begin{proof} 
Let $(x_n)_{n<\omega}$ be an enumeration of the elements of the countable 
lattice $T$. Let $(\mathcal D_n)_{n<\omega}$ and $(\mathcal
C_n)_{n<\omega}$ be the sequences of subsets of $\mathcal{C}_{ L}(T)$
defined inductively by the following conditions: for all $k < \omega$
\begin{align*}
\mathcal D_{0} &:= \emptyset ;\\
\mathcal C_k &:=\{C\in \mathcal C_{L}(T): x_k\in C\}\setminus \mathcal D_k ;\\
\mathcal D_{k+1} &:=\mathcal D_k\cup \mathcal C_k.
\end{align*}

\noindent
Then $\mathcal C_0=\{C\in \mathcal{C}_{ L}(T): x_0\in C\}$, $\mathcal
D_n:= \bigcup_{i < n} \mathcal{C}_i$ for $n<\omega$, and $T=\bigcup_{n
< \omega} \mathcal{D}_n$.

We wish to define a map $\varphi: \mathcal{C}_L (T) \rightarrow T$ such that for each $n<\omega$:

\begin{enumerate} [{(i)}]
\item \label{eq:condition1}$\varphi (C)\in C$
for all $C\in \mathcal D_n$;
\item \label{eq:condition2}$\varphi_{\restriction D_n}$ takes only finitely many values and 
\item \label{eq:condition3} $\varphi_{\restriction D_n}$ is order preserving.  
\end{enumerate}

\noindent
Set $\varphi (C):=x_0$ for $C\in \mathcal C_0$. Hence,
(\ref{eq:condition1}) - (\ref{eq:condition3}) hold for $n\leq 1$.  Let
$k\geq 1$ and suppose that $\varphi$ is defined on $\mathcal D_k$ such
that (\ref{eq:condition1}) - (\ref{eq:condition3}) hold for
$n=k$. Extend $\varphi$ on $\mathcal D_{k+1}$ as follows.  For
$C\inÊ\mathcal C_k$, set
\begin{align*}
\varphi_k^{+}(C) &= \ \bigwedge \{\varphi(D): D\in \mathcal D_k\;   \text{and}\; C<D \}, \\
\varphi_k^{-}(C)  &= \ \bigvee \{\varphi(D): D\in \mathcal D_k\;   \text{and}\; D<C \}, \ \text{and}\\
\varphi (C) &= \ (x_k\wedge\varphi_k^{+}(C))\vee\varphi_k^{-}(C). 
\end{align*}
Clearly, (\ref{eq:condition1}) and (\ref{eq:condition2}) hold for
$n=k+1$, so it remains to verify (\ref{eq:condition3}).  Let $C',
C''\in \mathcal D_{k+1}$ with $C'\leq C''$.  We consider three
cases:

\medskip

\noindent{\bf Case 1.} $C', C''\in \mathcal C_{k}$. 
Since (\ref{eq:condition3}) holds for $n=k$, $\varphi_{k}^{+}(C')\leq
\varphi_{k}^{+}(C'')$ and $\varphi_{k}^{-}(C')\leq
\varphi_{k}^{-}(C'')$. This yields $\varphi(C')\leq \varphi(C'')$.\\

\medskip

\noindent{\bf Case 2.} $C'\in \mathcal C_{k}$ and $C''\in \mathcal D_{k}$. 
Since (\ref{eq:condition3}) holds for $n=k$, $\varphi_{k}^{+}(C')\leq
\varphi(C'')$ and $\varphi_{k}^{-}(C')\leq \varphi(C'')$. Since
$\varphi (C')\leq \varphi_{k}^{+}(C')\vee \varphi_{k}^{-}(C')$,we get
$\varphi(C')\leq \varphi(C'')$.\\

\medskip

\noindent{\bf Case 3.} $C'\in \mathcal D_{k}$ and $C''\in \mathcal C_{k}$. 
Since (\ref{eq:condition3}) holds for $n=k$, $\varphi (C')\leq
\varphi_{k}^{-}(C'')$. Since by definition $\varphi_{k}^{-}(C'')\leq
\varphi (C'')$ we get $\varphi (C')\leq \varphi(C'')$. \\
 
\noindent The map $\varphi$ is an order preserving selection
map. This proves the theorem. \end{proof}

\subsection{Well-quasi-ordered posets, posets with no infinite antichain and a proof of Main Theorem \ref{thm:main2}} 
As said in the introduction, the cornerstone of Main Theorem
\ref{thm:main2} is implication $(iv)\Rightarrow (i)$.  Here is a
restatement.

\begin{lemma}\label{lem:key}
If $P$ has no infinite antichain, then $\mathcal {I}(P)$ has $CLSP$.
\end{lemma}

\noindent
The proof of this result follows from a Hausdorff type result due to
Abraham {\it et al.} \cite{abraham-all}.  In order to state this
result, first we recall that if $(P_{\alpha})_{\alpha \in A}$ is a family of
ordered sets indexed by a poset $A$, the \emph{lexicographical sum of
the $P_{\alpha}$'s indexed by $A$} is the poset, that we denote by
$\sum_{\alpha \in A}P_{\alpha}$, obtained by replacing each element
$\alpha\in A$ by $P_{\alpha}$ and by ordering the elements
accordingly. Finally, an \emph{augmentation} of a poset $P$ is a poset
$Q$ on the same underlying set such that $x \le y$ in $P$ implies $x
\le y$ in $Q$.

Let $\mathcal {BP}$ be the class of posets $P$ such that $P$ is either
a wqo poset, the dual of a wqo poset, or a linear ordering.  Let
$\mathcal P$ be the smallest class of posets such that

\begin{enumerate} [{(a)}]
\item $\mathcal P$  contains $\mathcal {BP}$;
 \item $ \mathcal P$ is closed under lexicographic sums with index set in $\mathcal {BP}$;
\item $\mathcal P$  is closed under augmentation.
\end{enumerate}

\noindent
The Hausdorff type result  is the following:\\

\begin{theorem}\label{main-abraham}\cite{abraham-all} 
$\mathcal P$ is the class of posets with no infinite antichain.
\end{theorem}

With this result in hand, we prove Lemma \ref{lem:key} as follows. We
start with $P\in \mathcal P$ and we prove that $\mathcal {I}(P)$ has
$CLSP$ by induction, distinguishing the following cases:

\begin{description}
\item [Case  1]  $P\in \mathcal{BP}$.

\item [Case  2]  $P$ is a lexicographic sum of posets  $P_{\alpha}$  indexed by a poset $A\in\mathcal {BP}$ such 
that $P_{\alpha}\in \mathcal {P}$ and $\mathcal {I}(P_{\alpha})$ has  $CLSP$ for each $\alpha \in A$.

\item [Case 3] $P$ is an augmentation of $Q$ such that $Q\in \mathcal P$ and $\mathcal {I}(Q)$ has  $CLSP$.
\end{description}

\noindent
{\bf Case  1}. If $P$ is a chain, then $\mathcal {I}(P)$ is a chain, which has $CLSP$ according to 
Proposition  \ref{prop:CLSPchain}. If $P$ is well-quasi-ordered  then $\mathcal {I}(P)$ is well-founded 
\cite {higman}, the fact that  it has  $CLSP$  is a consequence of the following lemma.

\begin{lemma}\label {lemma:well-founded} Every well founded lattice has  $CLSP$. 
\end{lemma}

\begin{proof}
Let $T$ be a well-founded lattice. Each nonempty sublattice $T'$ has a
least element $min(T')$.  The map $s: \mathcal{C}_{ L}(T)\rightarrow
T$ defined by $s(T'):= min(T')$ is order preserving, hence $T$ has $CLSP$.
\end{proof} \\

If $P$ is dually well-quasi-ordered, note that $\mathcal {I}(P)$ is
order isomorphic to $\mathcal {I}(P^*)^*$.  From above, $\mathcal
{I}(P^*)$ has $CLSP$. Since this property is preserved by duality,
it holds for $\mathcal {I}(P^*)^*$, hence for $\mathcal {I}(P)$.

\noindent
{\bf Case 2}.  The following lemma deals with the case of a
lexicographic sum indexed by a chain or a wqo.  The case of a
lexicographical sum indexed by a reverse wqo, use duality as in Case
1.  
 
\begin{lemma} \label{lexicosum}  
Let $A$ be a chain or a wqo poset, let $(P_{\alpha}), \alpha \in A$,
be a family of posets and let $T_{\alpha}:= \mathcal I (P_{\alpha})$
have $CLSP$ for each $\alpha \in A$.  Then $\mathcal {I}(\sum_{\alpha
\in A}P_{\alpha})$ has the $CLSP$.
\end{lemma}

\begin{proof}
According to Case 1 above, $T := \mathcal{I}(A)$ has $CLSP$.  Let
$s: \mathcal{C}_{ L}(T)\rightarrow T$ and $s_{\alpha}:\mathcal{C}_{
L}(T_{\alpha})\rightarrow T_{\alpha}$, $\alpha \in A$, be selection
maps.  Let $P:= \sum_{\alpha \in A}P_{\alpha}$ and $p: P\rightarrow A$
be the projection map.

Let $T'$ be a nonempty convex sublattice of $\mathcal {I}(P)$ and let
$\theta (T'):= \{p[I]: I\in T'\}$.  Since $\theta (T')$ is convex,
down-directed and up-directed, it is a nonempty convex sublattice of
$T$ and $A':=s(\theta (T'))$ is well defined.  Then $A' \in
\theta(T')$ so there is some $I' \in T'$ such that $p[I'] = A'$.

Let $\max A'$ be the set of maximal elements of $A'$ (perhaps the
empty set).  For each $\alpha \in \max A'$, let $T'_{\alpha}:= \{I\cap
P_{\alpha} : I \in T' \}$.  Routine checking verifies that
$T'_{\alpha}$ is a convex sublattice of $T_{\alpha}$.  Let
$s_{\alpha}(T'_{\alpha}) = P'_{\alpha}$ and define $\varphi$ on
$\mathcal{C}_{ L}(\mathcal{I}(P))$ as follows: $$\varphi (T')= \bigcup
\{P_{\alpha}: \alpha \in A' \setminus \max A' \} \cup
\bigcup\{P'_{\alpha}: \alpha \in \max A' \} .$$

\begin{claim} 
The map $\varphi :\mathcal{C}_{ L} (\mathcal{I}(P))\rightarrow
\mathcal{I}(P)$ is an order preserving selection map.
\end{claim} 

\noindent {\bf Proof of the Claim.} 
First let us see that $\varphi(T')\in T'$.  It is obvious that
$\varphi(T') \in \mathcal{I}(P)$.  If $\max A' = \emptyset$ then the
conclusion is immediate since $\bigcup \{P_{\alpha}: \alpha \in A'\}$
is the only element $I \in T'$ with $p[I] = A'$.

Let us assume that $\max A' \ne \emptyset$.  Then $\bigcup
\{P_{\alpha}: \alpha \in A' \setminus \max A' \}\subseteq I$ for each
$I\in T'$ with $p[I] = A'$.  For each $\alpha \in \max A'$, select
$I_{\alpha} \in T'$ such that $I_{\alpha}\cap P_{\alpha}=P'_{\alpha}$.
Since $A$ has no infinite antichain, $\max A'$ is finite and, hence,
these two initial segments belong to $T'$: $$ \bigcap ( I_{\alpha} :
\alpha \in \max A' ) \cap I'  \quad \text{and} \quad \bigcup ( I_{\alpha}:
\alpha \in \max A' ) \cup I' .$$

It is easy to see that $\varphi(T')$
is between these two sets in $\mathcal{I}(P)$ and so belongs to the
convex sublattice $T'$.

Now, let $T', T''\in \mathcal{C}_{ L}(\mathcal{I}(P))$ with $T'\leq
T''$. We check that $\varphi(T')\subseteq \varphi(T'')$.  Let $x\in
\varphi(T')$.  Then $x\in P_{\alpha}$ for some $\alpha \in A$. If
$\alpha$ is not maximal in $A'$, defined as above, then since we have
$A'\subseteq A''$, $\alpha$ is not maximal in $A''$, thus $x\in
\varphi(T'')$. If $\alpha$ is maximal in $A'$ and not maximal in $A''$
then the same conclusion holds.  The only remaining case is $\alpha$
is maximal in $A'$ and $A''$.  In this case, $$x \in P'_{\alpha} =
s_{\alpha}(T'_{\alpha}) \subseteq s_{\alpha}(T''_{\alpha}) =
P''_{\alpha} \subseteq \varphi(T'').$$ This depends on the fact that
$T'_{\alpha}\leq T''_{\alpha}$, a consequence of $T'\leq T''$.
\endproof \\

With that, the proof of the lemma is complete.\end{proof} \\

\noindent
{\bf Case 3}.  If $P$ is an augmentation of a poset $Q$ then $\mathcal {I}(P)$ is embeddable in 
$\mathcal {I}(Q)$; if $\mathcal {I}(Q)$ has the $CLSP$ then  Lemma \ref{cor:retract} asserts that 
$\mathcal {I}(P)$ has  the $CLSP$.\\

\section{Conclusion and further developments}\label{conclusion}

The following table summarizes various preservation  properties under
operations relevant to this investigation, and mentions one remaining
open question.

\begin{center}
\tiny
\begin{tabular}{|c|c|c|c|}\hline
  & \multicolumn{3}{c}{\bf Property of a lattice T} \\ \hline
{\bf Preservation under}  &  CLSP &  CLFPP &  $\mathcal C_{L}(T)$ is complete  \\ \hline
  Retracts & yes: Lemma \ref{lem:CLSPretraction} & yes: Lemma \ref{lem:CLFPPretraction}  & yes: Corollary \ref{cor:CL_{*}(T)retract+quotient}\\ \hline
  Quotient & yes: Corollary \ref{cor:CLSPquotient} & yes: Lemma \ref{lem:CLFPPretraction} & yes: Corollary \ref{cor:CL_{*}(T)retract+quotient}  \\ \hline
  Finite product & yes: Lemma \ref{lem:CLSPproduct} & Open    &    yes: Corollary \ref{cor:finiteproduct} \\  \hline

  \hline
\end{tabular}
\end{center}

\medskip

To be specific we pose the following:

\begin{problem} 
Is CLFPP preserved under finite products? 
\end{problem}

In Lemma \ref{lem:quotient-retract}, we showed that every quotient of
a lattice with $CLSP$ is a retract of that lattice. We do not know if
the corresponding  statement is true for CLFPP and $\mathcal C_{L}(T)$ being
complete. We therefore ask:

\begin{problem} 
\begin{enumerate}
\item Is every quotient of a lattice with $CLFPP$  a retract of that lattice?
\item Is every quotient of a lattice $T$ such that $\mathcal C_{L}(T)$ is complete  a retract of that lattice?
\end{enumerate}
\end{problem}

\bigskip

In this paper we have considered multivalued maps defined on a poset
$P$ whose values belong to a particular subset $\mathcal D$ of
$\powerset (P)$, this set being either $\mathcal C (P)$, or $\mathcal
C_L(T)$ provided that $P$ is a lattice $T$. In these two cases, we
have tried to relate the fixed point property for those maps to the
order structure of $\mathcal D$. There are other sets $\mathcal D$ to
consider, notably the set of \emph{bounded} sets, sets of the form
$S(\mathcal A, \mathcal B)$ where $(\mathcal A, \mathcal B)$ is a
separable pregap of $P$. One could rather consider other structures
than posets. Metric spaces seems to be appropriate. If we look at a
metric space $(E,d)$ as an object similar to a poset, the Hausdorff
distance on $\powerset (E)$ is the analog of the preorder we
defined on the power set of a poset, and the \emph{non-expansive} maps
$f$ from $E$ to $ \powerset(E)$ (satisfying $d(f(x),f(y))\leq d(x,y)$
for all $x,y\in E$) the analog of the order preserving maps. In the
theory of metric spaces, the spaces analogous to complete lattices seems to
be the \emph{hyperconvex metric spaces} introduced by N.Aronszajn and
Panitchpakdi in 1956, e.g. bounded hyperconvex metric spaces have FPP
(R.Sine, P.M.Soardi, 1979). Also, analogs of convex sublattices seem
to be the up directed unions of intersections of balls (for
hyperconvex spaces see \cite{kirk} and for the analogies between
posets and metric spaces see \cite{jawhari}).

\medskip 
We conclude with a specific problem in this direction. 

\begin{problem} 
For which metric spaces does  every multivalued map into the set of up
directed unions of intersections of balls have a fixed point?
\end{problem}



\begin{thebibliography}{10pt}
\bibitem{abraham-all} 
U.~Abraham, R.~Bonnet, J.~Cummings, M.~D\v{z}amondja, K.~Thompson, 
\newblock A scattering of orders, 26pp, 2010 to appear TAMS.
\bibitem{birkhoff}
G.~Birkhoff, M. K.~Bennett,
\newblock The convexity lattice of a poset. Order 2 (1985), no. 3, 223Ð242.

\bibitem{chakir-pouzet}
I.~Chakir, M.~Pouzet, 
\newblock Infinite independent sets in distributive lattices, Algebra Universalis {\bf 53} (2005) 211-225. 

\bibitem{davis}
A.~Davis,  
\newblock A characterization of complete lattices.  Pacific J. Math.  5  (1955), 311--319. 
\bibitem{dean}
 R. A.~Dean, 
 \newblock Sublattices of free lattices, in Proc. Symp. Pure Math. (R. P. Dilworth, ed.) Amer. Math. Soc., Providence, 1961.
\bibitem{kakutani} 
S.~Kakutani, 
\newblock A generalization of Brouwer's fixed point theorem.  Duke Math. J.  8,  (1941), 457Ð459.
\bibitem{D-P} 
D.~Duffus, M.~Pouzet, 
\newblock   Representing ordered sets by chains. 
Orders: description and roles (L'Arbresle, 1982), 81--98, 
North-Holland Math. Stud., 99, North-Holland, Amsterdam, 1984. 


\bibitem{D-R-S}
D.~Duffus, I.~Rival, M.~Simonovits,  
\newblock Spanning retracts of a partially ordered set. 
Discrete Math. 32 (1980), no. 1, 1--7.

\bibitem{D-R}
D.~Duffus, I.~Rival,
\newblock A structure theory for ordered sets. 
Discrete Math. 35 (1981), 53--118.


\bibitem{G} 
G.~Gr\"{a}tzer, 
\newblock General Lattice Theory, 
Birkh\"auser, 1998.


\bibitem  {higman}
G.~Higman, 
\newblock Ordering by divisibility in abstract 
algebras, Proc. London. Math. Soc. 2 (3),
(1952), 326-336.
\bibitem {jawhari}
E. ~Jawhari, D.~Misane, M.~Pouzet,
\newblock Retracts:  graphs and ordered sets from the metric point of view, Contemporary Mathematiccs, Vol 57, 1986, 175-226.
\bibitem{jech}
 T.~Jech, 
\newblock Set Theory, Springer monographs in mathematics, - 3rd Millennium ed, rev. and expanded. Springer, 2002. 
\bibitem{kirk}
W.A~Kirk and B.~Sims, Handbook of metric fixed point theory, Kluwer Academic Publishers, Dordrecht, 2001. xiv+703 pp.
\bibitem {L-M-P1} 
J.D.~Lawson,M.~Mislove, H.A.~Priestley,
\newblock  Infinite antichains in semilattices.  Order  2  (1985),  no. 3, 275-290. 


\bibitem{L-M-P2} 
J.D.~Lawson,M.~Mislove, H.A.~Priestley, 
\newblock Infinite antichains and duality theories.
Houston J. Math. 14 (1988), no. 3, 423-441. 
\bibitem{P-R}
M.~Pouzet, I.~Rival, 
\newblock Quotients of complete ordered sets.
Algebra Universalis 17 (1983), 393-405.


\bibitem{pou-riv84}
M.~Pouzet, I.~Rival, 
\newblock Every countable lattice is a retract of a direct product of chains. 
Algebra Universalis 18 (1984), no. 3, 295--307.

\bibitem {rival}
I.~Rival,
\newblock  A fixed point theorem for finite partially ordered sets. J. Combinatorial Theory Ser. A 21 (1976), no. 3, 309Ð318.
\bibitem{R-W}
I.~Rival, R.~Wille, 
\newblock The smallest order variety containing all chains. 
Discrete Math. 35 (1981), 203--212.

\bibitem{R}
A.~Rutkowski, 
\newblock Multifunctions and the fixed point property for product of ordered  sets. 
Order 2 (1985), 61-67.

\bibitem{semenova}
M.~Semenova, F.~Wehrung, 
\newblock Sublattices of lattices of order-convex sets. I. The main representation theorem. J. Algebra 277 (2004), no. 2, 825Ð860. 

\bibitem{smithson}
R. E.~ Smithson,  
\newblock Fixed points of order preserving multifunctions. Proc. Amer. Math. Soc.
28 (1971),  304-310.
\bibitem{tarski}
A.~Tarski, 
\newblock A lattice-theoretical fixpoint theorem and its applications.
Pacific J. Math. 5 (1955), 285--309.

\bibitem{walker}J.W.~Walker, 
\newblock Isotone relations and the fixed point property for posets. Discrete Math. 48 (1984), 275-288.
\end{thebibliography}
\end{document}